\documentclass[11pt]{article}
\usepackage{amsmath}
\usepackage{graphics}                                     
\usepackage{epsfig}                                       
\usepackage{graphicx}
\usepackage{epstopdf}       
\usepackage{amscd}
\usepackage{amsthm}
\usepackage{latexsym}
\usepackage{amsfonts}
\usepackage{tikz}

 \usepackage[all]{xy}

\newtheorem{theorem}{Theorem}[section]

\newtheorem{lemma}[theorem]{Lemma}
\newtheorem{proposition}{Proposition}

\theoremstyle{definition}

\newtheorem{remark}{Remark}

\newcommand{\ti}[1]{\tilde{#1}}

\def\d{\partial}
\def\AA{\mathbb{A}}
\def\X{\mathcal{X}}

\def\na{\nabla}
\def\G{\Gamma}

\def\ps{\psi_{\mbox{\tiny{$\na$}}}}

\def\pu{{\alpha}}

\def\bea{\begin{eqnarray}}
\def\eea{\end{eqnarray}}
\def\bl{\begin{lemma}}
\def\el{\end{lemma}}
\def\br{\begin{remark}}
\def\er{\end{remark}}
\def\a{\alpha}	

\def\bna{\tilde{\nabla}}

\newcommand{\R}{\mathbb{R}}
\newcommand{\sfi}{\sin{\phi}}
\newcommand{\cte}{\cos{\theta}}
\newcommand{\ste}{\sin{\theta}}
\newcommand{\cfi}{\cos{\phi}}

\title{About simple variational splines  \\from the Hamiltonian viewpoint}
\date{} 
\begin{document}

\maketitle

\vspace{-1cm}

\centerline{\scshape Paula Balseiro }
{\footnotesize
 \centerline{Departamento de Matem\'atica Aplicada,  Universidade Federal Fluminense
}
   \centerline{Rua M\'ario Santos Braga, S/N, Campus do Valonguinho }
   \centerline{ 24020-140, Niter\'oi, RJ, Brazil } 
   \centerline{pbalseiro@vm.uff.br}
} 

\smallskip

\centerline{\scshape Teresinha J. Stuchi}
{\footnotesize
 \centerline{Departamento de F\'isica Matem\'atica,  Universidade Federal do Rio de Janeiro}
   \centerline{ Centro de Tecnologia - Bloco A - Cidade Universit\'aria - Ilha do Fund\~ao}
   \centerline{21941-972 Rio de Janeiro - RJ - Brazil }
   \centerline{email{tstuchi@if.ufrj.br} 
}

\smallskip

\centerline{\scshape Alejandro Cabrera }
{\footnotesize
 \centerline{ Departamento de Matem\'atica Aplicada,  Universidade Federal do Rio de Janeiro}
   \centerline{ Centro de Tecnologia - Bloco C - Cidade Universit\'aria - Ilha do Fund\~ao}
   \centerline{21941-909 Rio de Janeiro - RJ - Brazil }
   \centerline{acabrera@labma.ufrj.br} 
   }

\smallskip

\centerline{\scshape Jair Koiller}
{\footnotesize
 \centerline{Instituto Nacional de Metrologia, Qualidade e Tecnologia}
   \centerline{Divis\~ ao de Metrologia em Din\^amica de Fluidos} 
\centerline{25250-020, Xer\'em, Duque de Caxias - RJ - Brazil } 
\centerline{jairkoiller@gmail.com}
}


\smallskip

{\footnotesize \begin{abstract}
In this paper, we study simple splines on a Riemannian manifold $Q$ from the point of view of the Pontryagin maximum principle (PMP) in optimal control theory. The control problem consists in finding smooth curves matching two given tangent vectors with the control being the curve's acceleration, while minimizing a given cost functional. We focus on cubic splines (quadratic cost function) and on time-minimal splines (constant cost function) under bounded acceleration. We present a general strategy to solve for the optimal hamiltonian within the PMP framework based on splitting the variables by means of a linear connection. We write down the  corresponding hamiltonian equations in intrinsic form and study the corresponding hamiltonian dynamics in the case $Q$ is the $2$-sphere. We also elaborate on possible applications, including landmark cometrics in computational anatomy.
keywords: Riemannian splines,  Computational anatomy, \\
$\mbox{     \,\,\,\,\,\,\,\,\,\,\,\,\,\,\,    \,\,\,\,\,    }$  Geometric control, Reduction, Reconstruction
   \end{abstract}
   }


  \newpage

{\small
 
   \tableofcontents

} 


\section{Introduction}

  Cubic Riemannian splines and their  higher order  extensions   have been instrumental  for  interpolation and statistics on manifolds.   The  bibliography is vast,  see e.g.  \cite{Popiel2007}, \cite{Hinkle2012}, \cite{Hinkle2014}, \cite{Burnett2013}, 
  \cite{Gay-Balmaz2012a}, \cite{Gay-Balmaz2012},  \cite{Niethammer2011}, \cite{Steinke2010}, \cite{Fletcher2013}.    In this note we consider only the simplest case, namely, of  splines having the tangent bundle as state space, the acceleration vector  being the control. 
We remark that a  tangent vector gives a minimal model  for a  {\it short  process.}   
For instance, the {\it rendezvous problem} in robotics and in  space science (\cite{Francis}, \cite{Francisshort}, \cite{LinFrancis}) consists of 
  planning  a path  with prescribed initial and end tangent vectors.  For instance, to achieve a smooth docking of a service spacecraft to the International Space Station. 
  
  In computational anatomy  \cite{Younes}   splines are  useful for longitudinal medical studies  \cite{Fletcher2004}, \cite{Singh2014},\cite{Singh2015}, \cite{Fiot2014},   
  combined with adaptive  machine learning techniques,  see e.g. \cite{Muralidharan2014}, \cite{Wang2014}.
   However,  the idea of comparing two  short physiological  processes seems
    not much explored as yet. This question  is important in embryology, where it is called {\it morphokinetics}   \cite{Desai:2014rr}.  
   
   There is a potential use of simple splines also in  sport science, computer animation, recognition, and video/movies repair, see \cite{Bauerarxiv}  and references therein.

 Let us now present our general framework, which belongs to the class of so-called  {\it  mechanical control problems}.
 These were   first studied  via Geometric Mechanics   in  Andrew Lewis thesis \cite{lewisthesis}, consolidating earlier work by several authors, specially R. Brocket, J. Baillieul,  A. Van der Schaft, P. Crouch and A. Bloch.  
  The standard reference is the book by Bullo and Lewis  \cite{LewisBullo}, where the  Lagrangian viewpoint  is mostly used. 

 One considers a configuration space $Q$ endowed with a Riemannian (kinetic energy) metric $g$.
The organism or device under study is  controlled by  forces that produce an acceleration $u$, taken here as the control variable.  If $\nabla$ denotes the Levi-Civita connection, the state equation on the tangent bundle  $TQ$  is
  \begin{equation} \label{stateequation}
  \nabla_{\dot{q}} \dot {q} = u\, .
  \end{equation}
\noindent  The aim is to connect two tangent vectors (or 'short processes') $v_{q_0},w_{q_1} \in TQ$  minimizing a cost functional. In this paper, we focus on the following two special cases which have received special attention.\\

  {\it    Cubic splines on  a Riemannian manifold}, which were introduced around 1990 (\cite{Noakes1989,CrouchLeite1991}), correspond to minimizing the cost functional: 
\begin{equation}  \label{cubicsplines}  \int_0^T\,\frac{\beta}{2}\, |u|^2\, dt,   
\end{equation}  with prescribed time $T$.  
Cubic splines have been extensively used in computational anatomy.

The other special case is the {\it time minimal  problem under bounded acceleration}. 
It  consists in connecting two vectors in minimum time,  under the restriction 
\begin{equation} \label{timemin}  |u| \leq A \,,\,\, {\rm where}\, A \,  {\rm is}\,  {\rm a}\,  {\rm prescribed} \, {\rm bound.}
\end{equation}
There is no  accessibility issue:  for any arbitrarily small  bound $\epsilon$ on $|u|$, under mild hypothesis on $Q$  it is possible to concatenate any two tangent vectors by a smooth curve with non vanishing velocity and acceleration norm $\leq \epsilon$ \cite{Weinstein1968}.  
The time minimal-bounded acceleration problem  is   in general (though not always)  equivalent to the   $L^{\infty}$ control problem considered recently by Noakes and  Kaya   \cite{KayaNoakes2013}, \cite{Noakes2014}, where they ask for a trajectory that minimizes the  sup  of  the norms of the accelerations, with fixed transition time. \\

From now on we shall refer to cubic splines also as \emph{$L^2$-splines}, due to the form of the underlying cost functional \eqref{cubicsplines}, and to time minimal-bounded acceleration splines as \emph{$L^{\infty}$-splines}, due to the previous discussion.\\

 Via Pontryagin's maximum principle, every  cost functional associated to  (\ref{stateequation}) yields a Hamiltonian system in
  $\,\, T^* (TQ)$.  
  In this paper, we study the resulting Hamiltonian systems in detail.

\subsection{Motivations} \label{motivation}

\subsubsection*{ Splines in $S^m$.} 
According to Lyle Noakes,  ``the problem of interpolating and approximating spherical data in the
m-dimensional manifold $S^m$ is much more widespread than might at first be   thought''  \cite{Noakesspherical2006}.    Indeed, see \cite{Watson}, \cite{Fisher}, 
\cite{bingham1985}, \cite{Dryden:2005rm} for information on spherical statistics.

In order to  match two infinitesimal processes on  $TS^m$ with simple variational splines, it seems to us that it is sufficient to understand  the cases  $m = 1,2,3$.  This is due to homogeneity under $SO(m+1)$ action:   
any two tangent vectors   $v_{q_1}, w_{q_2}$  on  `big' $TS^N$  are actually tangent to a isometrically imbedded $S^m$ with $ m \leq 3$.
 
 (For the analogous problem on $\R^N$,   $q_1$ can be the origin,  $v_1$,    $q_1 -  q_2$ and $v_2$  define at most a three dimensional subspace.)
    
  We have collected a number of references on cubic splines on $S^2$.
Somewhat surprisingly,   as far as we know, the reduction of the $SO(3)$-symmetry of  $T^*(TS^2)$  by Hamiltonian methods was  still awaiting.  We  present here  a reduction procedure, but we must confess that it works well only outside the zero section.

\subsubsection*{Splines in  {\bf Diff}}
The systems studied in this paper can be taken as finite-dimensional toy models for the following infinite-dimensional one.
Following the notations of \cite{HSS},  let $ \mathcal{D} $ a domain in $\R^d$ and    ${\rm Diff}(\mathcal{D}), \, \mathcal{X} $  respectively the group of diffeomorphisms $\phi$ and vector fields  ${\bf u}(x)$  with appropriate boundary conditions. The idea is to consider control systems as in \eqref{stateequation} but now on the infinite dimensional $Q={\rm Diff}$. Upon right translations, one has  $T {\rm Diff} = {\rm Diff} \times \mathcal{X}$,   $T^* {\rm Diff} = {\rm Diff} \times \mathcal{X}^*$. With due care to  functional analysis, $\mathcal{X}^*$ is the space of momentum densities  $\bf{m}$,  with  $m = \bf{m} \cdot {\bf dx}    \otimes   dV $. 
Consider a  Sobolev metric on  $ {\rm Diff} $ (\cite{Bruveristhesis}, \cite{harms},  \cite{Mich2Mum}, \cite{Bauer2015arXiv}, \cite{BaBruMi}, \cite{Bauer2011}, \cite{Bauer2012}, \cite{Bauer2016}). One can write  
\begin{equation}
L = \frac{1}{2}\, \int \, {\bf u}(x)\cdot  {\bf m}(x) dV \qquad \mbox{and} \qquad {\bf u}(x) =  \int \, G(x,y)\, {\bf m}(y) \, dV(y), 
 \end{equation}
where  $G(x,y)$ is the Green function for the inverse of the  linear partial differential operator that yields ${\bf m}$  when applied to  ${\bf u}$. 

EPDiff (geodesics on Diff),   the celebrated Euler-Poincar\'e partial differential equation  in terms of  the momentum density is
given by
\begin{equation}  \frac{\partial}{\partial t}\, \bf{m} +  \bf{u} \cdot \nabla \bf{m} + (\nabla \bf{u})^T \cdot \bf{m} + \bf{m}({\rm div}\, \bf{u}) = 0 \,\,\,\,\,\,  {\rm (EPDiff)},
\end{equation}
 and comes from a noble tradition going back to Arnold's interpretation  of Euler's incompressible fluid equations as geodesics in the infinite dimensional Lie group of volume preserving diffeomorphisms \cite{MumfordMichor2013}. 
One of the striking facts 
is that often the solutions of  EPDiff tend to concentrate on {\it pulson submanifolds}
\begin{equation*} {\bf m}(x,t) = \sum_{i}^N \, \int\, {\bf P}_i(s,t)\, \delta(x - {\bf Q}_i(s,t)).
\end{equation*} 
In turn, from   singular momentum solutions one recovers the vector fields via 
\begin{equation}  \label{u}  {\bf u}(x,t) =  \sum_{i}^N \, \int\, {\bf P}_i(s,t)\, G(x, {\bf Q}_i(s,t))\,ds \,\,. 
\end{equation}  
Mario Michelli  \cite{Micheli2012, Mich2Mum}  implemented   geodesic equations for   {\it landmarks}.
His formulae for for Christoffel symbols and curvatures coefficients in terms of cometrics,   can be used to  implement  a code for  landmark  splines using our methodology. In the infinite dimensional case 
cubic splines have  only recently been considered  \cite{Singh2014}, \cite{Singh2015}.    We will discuss the open problem of relating landmark splines to the infinite dimensional problem in the final section.

  \subsection{Contents of the paper}  
 In section \ref{usingconnection}, we study the control problem associated to the state equation \eqref{stateequation} from the point of view of Pontryagin maximum principle (PMP). The main theoretical results of the paper  are intrinsic formulas for the optimal hamiltonian and the resulting hamiltonian equations  (Propositions \ref{mainproposition1} and
  \ref{mainproposition2}). The key idea is to use \emph{split variables} for $T^*(TQ)$ coming from taking horizontal and vertical components w.r.t. the underlying linear connection on $TQ$.  
 A dual splitting  of $T (TQ)$ was  already explored  by Lewis and Murray and  is presented in full  details in Bullo and Lewis \cite{LewisBullosupp}.  In \cite{AlePoiJK} we generalize  the construction to $T^*A$, for $A \to Q$
 a vector bundle with a connection. 
 
 In split variables, the symplectic form is no longer canonical (it contains curvature terms, Proposition \ref{mainproposition1}) but the parametric hamiltonian appearing from the PMP has a simple form. We can then find the optimal hamiltonian for cubic and time-minimal splines easily and derive the corresponding hamiltonian equations in intrinsic form (Proposition \ref{mainproposition2}).
We show that, in the particular case $Q=S^n$ a $n$-sphere, we recover the higher dimensional case analogue to the hamiltonian system for cubic splines derived by Crouch and Leite \cite{CrouchLeite1991, CrouchLeite1995} for $S^2$.

In section \ref{observations}, we study dynamical aspects of the resulting hamiltonian equations on $T^*(TQ)$ in the particular case of cubic splines on $Q=S^2$.
Our main technical tool in this study is reduction by the natural rotation symmetry. The reduced system has two degrees of freedom. Two known  families of solutions  for cubic splines on $S^2$ are reinterpreted. They correspond to  equilibria  and  partial equilibria of  the reduced system. i)  ``Figure eights'' formed by two kissing circles with geodesic curvature $\kappa_g = 1$, run uniformly in time.  They   correspond to unstable, loxodromic, fixed points of the reduced Hamiltonian.   ii) Equators, run cubically on time,  correspond to ``partially fixed''  points of the reduced system.  
Moreover, in section \ref{simulations}  we discuss some numerical simulations. Poincar\'e sections indicate that the reduced  system  is non-integrable, but has regions rich of invariant tori. In section \ref{CrouchLeite1}, we discuss the limitations of the reduction procedure (coming from the rotation action not being free at zero tangent vectors) and provide the explicit relation to the general split variables approach.

Section \ref{conclusions}
contains some further comments  and presents suggestions for  further research.  In particular, since it is well known that    landmark geodesics
lift to geodesic solutions on  the full ${\rm Diff}$  \cite{Micheli2012},  \cite{BruverisHolm2015}, we pose the question if landmark splines can  be lifted to splines on ${\rm Diff}$. 

  In appendix \ref{program}, we present a simple Fortran program used for the reconstruction of trajectories in  $S^2$. It uses the Runge-Kutta-Fehlberg routine RK78  (\cite{Fehlberg}) which is of standard use in Celestial Mechanics  \cite{Attri}. In appendix \ref{convexsurfaces}, we discuss how to interpret the controls entering the state equation \eqref{stateequation} in terms of curvatures of the underlying curve, in the particular case $Q=\Sigma \subset \R^3$ is a convex surface. These results are applied in section \ref{observations}.

 \section{Optimal control on $TQ$ as state space}  \label{usingconnection}

As mentioned in the introduction, our main object of study will be the following optimal control problem. The state space is $TQ$, where $(Q,g)$ is a Riemannian manifold, and the control is represented by a (force or acceleration) vector field $u \in \mathcal{X}(Q)$. The state equations \eqref{stateequation} can be written in first-order form as
 \begin{equation}\label{eq:generalstateeq}
  v = \dot{x} , \ \ \ \nabla_{\dot{x}} v = u, \ \ (x(t),v(t)) \in TQ.
 \end{equation}
 Above, $\nabla_{\dot{x}}$ represents covariant derivative with respect to the Levi-Civita connection $\nabla$ on $TQ$ so that, for $u=0$, we get the equations for the geodesic flow on $TQ$. We impose the boundary conditions that $(x(t),v(t))$ is fixed to be given $(x_0,v_0),(x_f,v_f) \in TQ$ at initial and final times.
 
 To recover the second-order equation \eqref{stateequation}, we notice that the first equation in \eqref{eq:generalstateeq} above says that we are dealing with a second order problem for $x(t)\in Q$. More precisely, the above equations correspond to the second order vector field $U\in \X_{(2nd)}(TQ)$ given by
 $$ U|_{(x,v)} = hor_\nabla(v)|_{(x,v)} + u^{vert}|_{(x,v)} \in T_{(x,v)}TQ, $$
 where $hor_\nabla$ denotes the horizontal lift w.r.t. $\nabla$ and $u^{vert}$ the natural vertical lift $T_xQ \to T_{(x,v)}(TQ), a \mapsto \frac{d}{dt}|_{t=0} (v + t a)$. In standard local coordinates for $TQ$, the state equations reduce to
$$ \dot{x}^k = v^k \,\,,\,\, \dot{v}^k = -  \,
\Gamma^k_{ij} v^i v^j  +   u^k \,\,,\,\,  k = 1, \cdots n,$$
where $\Gamma^k_{ij}$ are the Cristoffel symbols of $\nabla$ and sum over repeated indices is understood.

Following the introduction further, we shall also consider an optimization component in the problem: $\gamma(t)=(x(t),v(t))$ must also minimize a cost functional of the form
$$ \gamma \mapsto \int_0^T C(U|_{\gamma(t)}) dt $$
with $C: T(TQ) \to \mathbb{R}$ a given cost function. The control $u$ is also (possibly) subjected to a constraint of the form 
$$ g(u,u) \leq A^2 \qquad \mbox{for a constant $A$}.$$

\subsubsection*{Applying Pontryagin's maximum principle.} Our general strategy to attack the above optimization problem will be to apply PMP and transform it into a hamiltonian system on $T^*(TQ)$. 

We  briefly describe here  the principle in our present situation, however we 
assume that the reader is acquainted with the general recipe for PMP   (otherwise we suggest \cite{Ross} for a tutorial book,  \cite{Pontryagin} for the fundamental reference in the area). 

The general key idea is, for each state variable, to introduce a new co-state variable. This leads one to consider $T^*(TQ)$ endowed with its canonical symplectic form $\omega_0$ and consider the $u$-family of hamiltonians $H_u \in C^\infty(T^*(TQ))$ given by
\begin{equation}\label{eq:Hu}
 H_u((x,v,P)) = - C(U|_{(x,v)}) + \langle P,  U \rangle, \ \ P\in T^*_{(x,v)} (TQ), 
\end{equation}
where $\langle \cdot, \cdot \rangle$ denotes the natural pairing between covectors and vectors. 
The PMP then states that the solution to our optimal control problem is a trajectory (for some suitable initial conditions to be found) of the hamiltonian system $(T^*(TQ), \omega_0, H_*)$ where 
$$ 
H_* := \underset{u}{max} \ H_u 
$$
is the optimal hamiltonian function.

In local coordinates, denoting $y_k$ and $z_k$ the conjugated coordinates to $x^k$ and $v^k$ respectively, we have
$$ H_u  \equiv - C(x,v,u) +   \,[ y_k v^k  + z_k  ( u^k -  \,
\Gamma^k_{ij} v^i v^j )].$$
It is easy to optimize this local hamiltonian in the $u^k$'s.
In general, though, it is not so easy to find the optimal hamiltonian in an intrinsic way (i.e. globaly w.r.t. the manifold $Q$). 

In the following subsections we propose a method to achieve this and to write down the corresponding hamilton equations in intrinsic form. This method is general and makes use of 'global splittings' of variables in which the $H_u$ becomes simple (and hence easier to optimize) but the symplectic structure is no longer in canonical form (it incorporates curvature terms). 

Other methods simplifying the optimization of $H_u$ can be available in particular cases, in which case we provide the dictionary between the two.

\subsection{Hamiltonian equations from PMP in split variables}\label{subsec:splitvar}

In this subsection, we introduce global split variables to solve the optimal hamiltonian described above. These results can be seen as the
Hamiltonian analogue of some results given in Lagrangian form  in the supplementary materials of  Bullo and Lewis' book 
\cite{LewisBullosupp}. 

 Moreover, a Hamiltonian version  goes back to   Crouch, Leite
and Camarinha \cite{CrLeCa},  Iyer \cite{Iyer2005, Iyer2006}, and   more recently on Abrunheiro et al.    \cite{abrunheiro2011, abrunheiro2013a, abrunheiro2014, abrunheiro2013b} and \cite{Valle2015}. 

 In section \ref{doublebundles} we  outline further developments:  a splitting for $T^*A$, with $A$ being an affine bundle with
 connection  \cite{AlePoiJK}.    
 
The main observation is that the linear connection $\nabla$ on $q:TQ \to Q$ allows one to decompose tangent vectors into horizontal and vertical components:
$$ X \in T_{(x,v)} (TQ) = Hor|_{(x,v)}\oplus Ver|_{(x,v)} \Rightarrow  X = X_h + X_v.$$ 
In more differential geometric terms, $\nabla$ induces an Ehresmann connection for the submersion $q: TQ \to Q$ given by the bundle projection. The vertical component of $X \in T_{(x,v)}(TQ)$ can be written as $X_v = \Theta_\na(X)^{vert}$ with $\Theta_\na \in \Omega^1(TQ,TQ)$ a vector valued 1-form encoding the vertical projection.

Since both horizontal and vertical spaces can be identified with $T_xQ$ by taking horizontal and vertical lifts,
the above decomposition defines a global diffeomorphism\footnote{Actually, it is a double vector bundle isomorphism.}
\begin{equation*}
 \begin{split}
  \phi_\na: q^*TQ & \oplus_{TQ} q^* TQ \to T(TQ)\\
 w & \oplus a|_{(x,v)}  \mapsto hor_\na (w)|_{(x,v)} + a^{vert}|_{(x,v)}.
 \end{split}
 \end{equation*}
The dual decomposition 
$$ T^*_{(x,v)}(TQ) \simeq Hor^*|_{(x,v)}\oplus Ver^*|_{(x,v)}\Rightarrow P = P_h + P_v\in T^*_{(x,v)}(TQ)$$ similarly induces
 a splitting diffeomorphism
$$ \psi_\na: q^*T^*Q \oplus_{TQ} q^*T^*Q \to T^*(TQ),$$
which is characterized by\footnote{Notice that $\psi_\na$ is the inverse of the dual, w.r.t. the projection $(w\oplus a)|_{(x,v)}\mapsto (x,v)$, of $\phi_\na$.}
$$ \langle \psi_\na(p \oplus \alpha)|_{(x,v)}, \phi_\na(w\oplus a)|_{(x,v)} \rangle = \langle p, w \rangle + \langle \alpha, a\rangle.$$ 
The covectors $p$ and $\alpha$ above define {\it split variables} 
\begin{equation}\label{eq:generalsplitvars}
 (x,v,p,\alpha) \equiv (p\oplus \alpha)|_{(x,v)} = \psi_\na^{-1}(P)
\end{equation}
for every  co-vector $P\in T^*_{(x,v)}(TQ)$.

For clarity, let us examine these diffeomorphisms in local coordinates. Let $(\tilde{p}_i,\tilde{\alpha}_j, \ti{v}^k, \ti{x}^k)$ be canonical coordinates in $T^*TQ$ relative to standard ones $(\tilde{v}^k,\tilde{x}^k)$ on $TQ$ and let $(p_i,\alpha_j, v^k, x^k)$ be natural coordinates on  $  q^*T^*Q \oplus_{TQ} q^*T^*Q$. Then,
\begin{equation}\label{eq:psicoords}
\psi_\na^* \tilde{p}_i = p_i + \Gamma_{ij}^k v^j \alpha_k, \ \psi_\na^*\tilde{\alpha}_j = \alpha_j , \ \psi_\na^*\ti{v}^k = v^k, \ \psi_\na^*\ti{x}^k = x^k. 
\end{equation}

The next proposition shows the effect of using global split variables in the symplectic form and the underlying general form of Hamilton's equations.

\begin{proposition}  \label{mainproposition1}  (Symplectic structure in split variables)
\begin{enumerate}
 \item[$(i)$]
 The pullback $\theta_\na : = \ps^* \theta_{TQ}$ of canonical 1-form  $ \theta_{TQ}\in \Omega^1(T^*(TQ))$    to the   split cotangent bundle yields
\begin{eqnarray}  \theta_\na|_{(x,v,p,\alpha)}   &=& \pi_1^* \theta_Q|_{(x,p)} + \pi^*_2\langle \alpha, \Theta_\na|_{(x,v)}\rangle \label{eq:splitoneform}
\\ & \overset{loc}{\equiv}& p_i dx^i + \pu_a (dv^a + \G_{ib}^a v^b dx^i),\nonumber 
 \end{eqnarray}
 where $\pi_1(x,v,p,\alpha) = (x,p)$ and $\pi_2(x,v,p,\alpha)=(x,v)$ denote the natural projections.
 \item[$(ii)$]  The pullback $\Omega_\na : = \ps^* \Omega_{TQ}$ of the canonical symplectic form $\Omega_{TQ}$ on $T^*( TQ)$ yields 
\begin{eqnarray}  \label{symplecticstructure}
\Omega_{\na}|_{(x,v,p,\alpha)} &=& \pi_1^*\Omega_Q|_{(x,p)} + \langle \pi_2^* \Theta_\na|_{(x,v)}  \wedge \pi_3^*\Theta_{\bna}|_{(x,\alpha)} \rangle -\pi_0^*  \langle \pu, R v\rangle|_x \label{eq:split2form} \\
&\overset{loc}{\equiv}& dx^i \wedge dp_i + (dv^a + \G_{ib}^a v^b dx^i) \wedge (d\pu_a - \G_{ja}^c \pu_c dx^j) \nonumber \\ 
&& - \frac{1}{2} R_{ija}^b v^a \pu_b dx^i \wedge dx^j ,\nonumber
\end{eqnarray}
where $\pi_0(x,v,p,\alpha)= x$, $\pi_3(x,v,p,\alpha) = (x,\alpha)$ denote natural projections, $R\in \Omega^2(M, End(TQ))$ is the Riemannian curvature tensor of $\na$ and $\Theta_{\bna} \in \Omega^1(T^*Q, T^*Q)$ corresponds to the vertical projection relative to the dual connection\footnote{The Christoffel symbols of $\bna$ are minus the transpose of those of $\na$, $\bna_{\d_{x^i}}dx^j=-\Gamma_{ik}^j dx^k$.} $\bna$ on $T^*Q \to Q$.
\item[$(iii)$]  Given $H\in C^\infty( q^*T^*Q \oplus_{TQ} q^*T^*Q)$,    the Hamiltonian vector field $X_H$  is given in local coordinates by\\
$$ \hspace*{-6cm} \dot{x}^i  =  \d_{p_i}H \,\,\,\,\,\, $$
\begin{equation}   \label{mainequationsa} \dot{p}_i  =  -\d_{x^i}H + \Gamma_{ia}^b (v^a \d_{v^b}H-\a_b \d_{\a_a}H)
			+ R^b_{ija} v^a \alpha_b \dot{x}^j 
			\end{equation} 
\begin{equation*} \hspace*{-0.8cm}  \dot{v}^a +
      \Gamma^a_{ib} \dot{x}^i v^b   =     \d_{\a_a}H, \quad 
 \dot{\alpha}^a -
      \Gamma^b_{ia} \dot{x}^i \alpha_b  = -\d_{v^a}H.
\end{equation*}
\end{enumerate}
 \end{proposition}
 \begin{proof}
 Both the l.h.s. and the r.h.s. of equations \eqref{eq:splitoneform} and \eqref{eq:split2form} define global differential forms on $ q^*T^*Q \oplus_{TQ} q^*T^*Q$. To prove $(i)$ and $(ii)$ it is then enough to show that, when restricted to any coordinate chart, the corresponding local expressions of the l.h.s. and of the r.h.s. coincide. Now, the formulas after the $\equiv$ symbols evidently correspond to the local coordinate expressions for the r.h.s'. Hence, we only need to show that the local expressions of $\psi_\na^*\theta_{TQ}$ and $\psi_\na^*\Omega_{TQ}$ coincide with the given ones.
 Let us then choose coordinates for $T^*TQ$ and $ q^*T^*Q \oplus_{TQ} q^*T^*Q$ as in eq. \eqref{eq:psicoords}. We have that
 $
 \theta_{TQ} \equiv \tilde{p}_i d\tilde{x}^i + \tilde{\alpha}_i d\tilde{v}^i
 $
 and thus $(i)$ follows directly by computing the pullback  $\psi_\na^*\theta_{TQ}$ following the change of coordinates \eqref{eq:psicoords}. For $(ii)$, we observe that
 $$
 \psi_\na^*\Omega_{TQ}= \psi_\na^*(-d\, \theta_{TQ})=-d (\psi_\na^*\theta_{TQ}) =  -d\, \theta_\na,
 $$
 so that we need to apply $-d$ to the known local expression for $\theta_\na$. By direct computation using the following local expression for the curvature tensor
 $$ \langle dx^\ell, R(\d_{x^i},\d_{x^j})\d_{x^k}\rangle = R_{ijk}^\ell(x) = \d_{x^i} \G_{jk}^\ell - \d_{x^j} \G_{ik}^\ell + \G_{jk}^\tau \G_{i\tau}^\ell - \G_{ik}^\tau \G_{j\tau}^\ell,$$
 one obtains the desired equality. Finally, $(iii)$ is a straightforward consequence of $(ii)$. 
 \end{proof}
 
  \bigskip
\noindent We recognize that the l.h.s'. of the last two equations in $(iii)$ correspond to covariant derivatives.  
{\it The equations for the $\dot{p}_i$ are more intricate, but we will show below that they  simplify for   
spline  control problems with  cost functions of a special form. }

\subsubsection*{Optimal spline Hamiltonians} 
Let us now show how the split variables simplify the hamiltonian of our problem. Indeed, recalling $H_u$ defined in \eqref{eq:Hu}, then
$$H_{u,\na}:=\psi_\na^* H_u  = - C(U|_{(x,v)}) + \langle \alpha, u(x) \rangle + \langle p , v \rangle $$
holds globally in the split phase space $q^*T^*Q \oplus_{TQ} q^*T^*Q$.
We shall restrict ourselves to the case in which $$C(U_{(x,v)}) = c(g_x(u(x),u(x)))$$ is a (typically convex) function $c: \R_+ \to \R$ of the norm square of the control $u$. Cubic splines have cost functional \eqref{cubicsplines} and are thus a particular case of the above with $c$ being a linear function. Time-minimal splines are also a particular case with  $c \equiv -1$ a constant (recall that, in this case, $u$ is constraint by \eqref{timemin}).

It is now easy to find the optimal value $H_{*,\na}$  of $H_{u,\na}$:
\begin{equation}     \label{splineshamilt}
  H_{*,\na}=  \langle p,v \rangle +{\rm Leg(c)}(g^{-1} ( \a, \a)) ,
  \end{equation}
where $g^{-1}$ is the cometric, the optimal value of the control is
  \begin{equation*} 
  u_* = \underset{u: \ {\rm constraints}}{\rm argmax} \,\, [\, \langle \alpha,u(x) \rangle -  c(|u_x|^2)\, ] , 
\end{equation*}
and ${\rm Leg(c)}$ denotes the {\it Legendre-Fenchel  dual} of $c$ \cite{Rockafellar}:  
\begin{equation} \label{eq:Legendretr}  
c(g(u_*,u_*)) +  {\rm Leg(c)}( g^{-1} (\alpha,\alpha)) =  \langle \alpha\,,\,  u_*\rangle \, . 
\end{equation}
 
 We now show that a dramatic simplification in the $\dot{p}_i$ equation results from  the connection preserving the metric. 

\begin{proposition}  \label{mainproposition2}  Hamilton's equations for $H_{*,\na}$ in the case of cost functions of the form  $C = c(g(u,u)) $ can be intrinsically written as 
  \begin{equation}  \label{intrinsic}
  \dot{x}  =   v \,\,\,,\,\,\,
 \ti{\nabla}_{\dot{x}}p   =  -i_v \langle \alpha, R v\rangle  
 \,\,\,,\,\,\,
 \nabla_{\dot{x}}v     =   u_*    \,\,\,,\,\,\, 
 \ti{\nabla}_{\dot{x}}\a  =  - p.
 \end{equation}
 Locally, they read
 $$ \dot{x}^i     =     v^i   ,\
   (\tilde{\nabla}_{\dot{x}} p)_i   =   \alpha_b  R^b_{ijk} v^j  v^k ,\
  (\nabla_{\dot{x}}v)^a  =  u_*^a  ,\ (\ti{\nabla}_{\dot{x}} \a)_a  = - p_a   .
  $$
\end{proposition}

 \begin{proof}
 From item $(iii)$ in Proposition \ref{mainproposition1} and the definition of $H_*\equiv H_{*,\na}$ given in \eqref{splineshamilt}, we immediately get
 $$
 \dot{x} = \d_{p}H_* = v \qquad \mbox{and}\qquad \bna_{\dot{x}} \alpha =    -\d_{v}H_* = - p.
 $$
 The equation for $\dot{v}$ reduces to
 $$ (\na_{\dot{x}}v)^i = \d_{\alpha_i} \left( {\rm Leg(c)}(g^{-1} ( \a, \a))\right).$$
 Deriving with respect to $\alpha_i$ both sides of the eq. \eqref{eq:Legendretr} defining ${\rm Leg(c)}$ one gets 
 $$ \d_{\alpha_i} \left( {\rm Leg(c)}(g^{-1} ( \a, \a))\right) = u_*^i$$
 as wanted. We are thus only left with the equation \eqref{mainequationsa} for $\dot{p}$. Transporting the term $ \Gamma_{ia}^b  v^a \d_{v^b}H_* =  \Gamma_{ia}^b v^a p_b$ to the l.h.s. to get a covariant derivative one obtains
 $$ (\bna_{\dot{x}}p)_i = - (\d_{x^i} H_* + \Gamma_{ia}^b \alpha_b \d_{\alpha_a} H_*) +   R^b_{ija} v^a \alpha_b \dot{x}^j.$$
 The proof will be finished when we show that the term between brackets in the r.h.s. vanishes. This, in turn, follows by virtue of the fact that $\nabla$ preserves the metric $g$, and hence $\bna$ preserves $g^{-1}\equiv (g^{ab}(x))$, so that 
 $$ (\ast) \ \ \d_{x^i} g^{ab} =  -  \Gamma^a_{ic} g^{cb} -  \Gamma^b_{ic} g^{ac}. $$
 Finally, this identity directly implies the desired vanishing:
 $$ \partial_{x^i} H_*+ \Gamma_{ia}^b \alpha_b \d_{\alpha_a} H_* = [{\rm Leg(c)}]' \ \d_{x^i} g^{ab} \a_a \a_b + 
 \Gamma_{ia}^b \alpha_b ( 2  [{\rm Leg(c)}]' g^{ac}\a_c) \overset{(\ast)}{=} 0. $$
\end{proof}

\subsection{Cubic and time minimal splines}\label{subsec:gensplinesspliteqs} 
Let us examine the particular cases mentioned in the introduction. For cubic splines, 
the cost functional is \eqref{cubicsplines} so 
$$  H_u(x,v,p,\alpha) = - \frac{\beta}{2} g(u,u) +   \langle \alpha , u \rangle  +  \langle p , v \rangle . $$
The advantage of the split variables  is now apparent: the optimal hamiltonian is immediately given by
\begin{equation}  \label{cubichamiltonian1}
H^{\rm cubic}:= H_{*,\na} = \frac{1}{2\beta} \, g^{-1} (\alpha , \alpha) +  \langle p , v \rangle \,\,,\,\, \,\,\,   u_* = \alpha^{\sharp}/\beta,
\end{equation}
where, for any covector $\alpha$, one defines $\alpha^\sharp$ by $g(\alpha^\sharp,v)=\alpha(v)$ for all vectors $v$.
Likewise,  the  optimal Hamiltonian for the time minimal problem with constraint \eqref{timemin}  is also easily shown to be
\begin{equation} \label{timeminimal1}
 H^{\rm tmin}:= H_{*,\na}   = -1 +  A \sqrt{ g^{-1} (\alpha , \alpha) } +  \langle p , v \rangle \,\,, \,\, u_* = A\, \alpha^{\sharp}/|\alpha^{\sharp}|.
\end{equation} 
The equations of motion in both cases are  (\ref{intrinsic})  with the corresponding  $u_*$  from (\ref{cubichamiltonian1}) and (\ref{timeminimal1}). 
For instance, taking $\beta=1$ in the cubic splines problem, then  $\nabla_{\dot x} \dot x = u_* = \alpha^\sharp$.  Deriving this equation covariantly two more times and using the equations of motion for $\alpha$ and $p$ we get \footnote{It is also useful to recall the identities $\nabla_{\dot{x}}(\a^\sharp)=(\bna_{\dot{x}}\a)^\sharp$ and $g(R(u,v)w,z)=g(R(w,z)u,v)$.}
\begin{equation}  \label{oldcubic} 
\nabla^{(3)}_{\dot x} \dot x = - R (\nabla_{\dot x} \dot x, \dot x) \dot x,
\end{equation}
recovering the equations found by  L. Noakes, G. Heinzinger and B. Paden \cite{Noakes1989},  and P. Crouch and F. S. Leite \cite{CrouchLeite1991}.

\subsubsection*{Landmark splines on $Q = (\R^d)^N$}   
  Before moving on to splines on spheres, we present some comments about landmark cometrics. For $N=1$ one has the euclidian metric on $\R^d$, for which  $L^2$  splines are  cubic polynomials on each coordinate.  The $L^{\infty}$ problem has been addressed   in \cite{CastroKoiller} for any value of $d$ (actually $d=2,3$ is enough).   The next simplest nontrivial case is $d=1, N=2$.
One observes that the underlying geodesic problem (i.e. when the control $u=0$) is integrable, with  Hamiltonian  
\begin{equation*}
 2H =  p_1^2 + p_2^2 + 2 G(x_1-x_2) \,p_1 p_2    \,\,.
 \end{equation*}
In the spline problem ($u\neq 0$) one also has invariance under translations on the line, so there will be a conserved momentum and it will be reducible to 3 degrees of freedom.  
Numerical  experiments for landmark splines are in order.  In ``Mario's formulas'',  partial derivatives are computed on cometric entries \cite{Micheli2012}. At every computation  step, done at the current landmark locations, 
  there is  only one matrix inversion,  of the  cometric matrix, which has a block structure.   For simulations we suggest  using the Cauchy kernel
$
G(x_1,x_2) =  1/(1+|x_1-x_2|^2)
$ that has a weaker decay at infinity than kernels involving the exponential.

\subsubsection*{Cubic splines on $Q=S^n$ and extrinsic vs intrinsic description}  \label{CrouchLeite2}
In this case, there is an alternative approach to finding the optimal hamiltonian which uses extrinsic variables coming from the embedding $S^n\subset \R^{n+1}$. We shall show below how to explicitly relate the two descriptions.

We first follow Dong-Eui Chang \cite{Chang2011}, fix a sphere $S^n(r)$ of radius $r$ and consider the following state equations in $\R^{2(n+1)}$ 
\begin{equation*}
\dot{x} = {\bf v}, \qquad  \dot{{\bf v}} = u - |{\bf v}|^2 \,x/r^2 , \qquad u  \perp x .
\end{equation*}
To avoid confusions with scalar velocities used later, we use {\bf boldface} for velocity \emph{vectors} from now on.  
The idea is that these equations have  $TS^n(r) \subset \R^{2(n+1)}$ as invariant submanifolds and induce the correct state equations \eqref{eq:generalstateeq} on the sphere. 

The natural coordinates on $\R^{2(n+1)}$ restrict to $TS^n(r)$ yielding (local, but almost global) coordinates that we call \emph{extrinsic variables} and denote by $(x,v)$.
Following the notation of section \ref{subsec:splitvar}, we consider the cotangent bundle $T^* (T \R^{n+1} ) =  \R^{4(n+1)}$ with coordinates  $(x,{\bf v}, \tilde{p},\, \tilde{\a})$ and canonical symplectic form 
\begin{equation} \label{canonical}
\Omega_{T \R^{n+1}} =    dx \wedge d\tilde{p} +d {\bf v}\wedge d\tilde{\a}.
\end{equation} 
In the case of cubic splines, the parametric Hamiltonian \eqref{eq:Hu} in the ambient $\R^{4(n+1)}$ is 
\begin{equation*}    
\hat{H}_u = - \beta ||u||^2/2 + \tilde{p} \cdot  {\bf v}    +  \tilde{\a} \cdot ( u - \frac{ |{\bf v}|^2 }{r^2} \, x),
\end{equation*}
 where the controls are restricted to the  tangent  planes:   $ u \perp x$. Notice that the ambient scalar product $\cdot$ allows us to identify vectors and covectors. Let us consider the projections
 \begin{equation}  \label{pvpar}  
\a^{\parallel} = \tilde{\a} - \langle \tilde{\a},x \rangle x /r^2 ,  \qquad \tilde{p}^{\parallel} = \tilde{p} - \langle \tilde{p},x \rangle x /r^2 
\end{equation}
onto the plane perpendicular to $x$, so that $(x,v,\tilde{p}^{\parallel},\a^{\parallel})$ define extrinsic variables for $T^*(TS^n(r))\subset \R^{4(n+1)}$. It is immediate to deduce that, upon restriction $H_u = \hat{H}_u|_{T^*(TS^n(r))}$,  the optimal control is
 $ u_* = \frac{1}{\beta} \a^{\parallel}\,\,\,$ 
and that the optimal Hamiltonian reads 
\begin{equation}  \label{direct}     
H_* = \frac{1}{2\beta}  |\a^{\parallel}|^2 + \tilde{p} \cdot  {\bf v}   - \, \frac{ \langle \tilde{\a} ,x \rangle }{r^2}   |{\bf v}|^2 = \frac{1}{2\beta}  |\a^{\parallel}|^2 + (\tilde{p}^{\parallel}  - \, \frac{\langle \tilde{\a} ,x \rangle}{r^2}\, {\bf v} )\, \cdot  \, {\bf v} .
\end{equation}
The relation between the extrinsic variables $(x,v,\tilde{p}^{\parallel},\a^{\parallel})$ and the split variables $(x,v,p,\a)$ of section \ref{subsec:splitvar} is given by the following: 
\begin{proposition}  \label{splittingsphere}  The split variables \eqref{eq:generalsplitvars} for  $T^*(TS^n)$  are given by
\begin{eqnarray*}  
\alpha & = &  \a^{\parallel} = \tilde{\a} - \langle \tilde{\a},x \rangle x /r^2  \in T^*_x S^n(r)  \,\,\, (\equiv T_x S^n(r) ) \\ 
  p & = & \tilde{p}^{\parallel}  - \, \frac{ \langle \tilde{\a} ,x \rangle}{r^2}\, {\bf v}  \,\,\, \in  T^*_x S^n(r)  \,\,\, (\equiv T_x S^n(r)\,).
\end{eqnarray*}
\end{proposition}
Of course, expressing the Hamiltonian \eqref{direct} in terms of the split variables we get
\begin{equation*}  
H_{*,\na} =  \frac{1}{2\beta}  |\alpha|^2 +  \langle p, {\bf v} \rangle ,
\end{equation*}
 which is the general optimal hamiltonian \eqref{cubichamiltonian1}. We stress the evident simplification operated on the Hamiltonian \eqref{direct} when passing to split variables.

\subsubsection*{Hamiltonian equations for cubic splines in $S^n$ and Crouch-Leite equations} To write down the equations of motion \eqref{intrinsic} in this particular case, let us first notice that for a curve $(x(t),w(t)) \in TS^n(r)$ 
described in extrinsic variables, we have 
$$ \na_{\dot{x}} w = \dot{w} + (w\cdot \dot{x}) x.$$
Indeed, the above implies $(\na_{\dot{x}} w) \cdot x =0$ by virtue of $w\cdot x = 0$, so that the covariant derivative remains tangent to the sphere.
Secondly, the curvature tensor of the sphere can be expressed in terms of the ambient inner product as
\begin{equation*}  
R(X,Y) Z =   (Y \cdot Z)X - (X \cdot Z)Y.
\end{equation*}  
Introducing the following change of notation 
\begin{equation*}  x = x_o\,\, ,\,\,\, {\bf v} = x_1\,\,\,,\,\,\, \alpha =    \a^{\parallel} = x_2\,\,\,,\,\,\,  p = p^{\parallel}  - \, \frac{\langle\tilde{\a} ,x \rangle }{r^2}\,{\bf v } =  - x_3
\end{equation*}
and recalling $u_*=\alpha^\sharp(\equiv \alpha)$ in the cubic spline case, it immediately follows that eqs. \eqref{intrinsic} yield:
\begin{eqnarray}   \label{CLn}
\dot{x}_o & = &   x_1  \,\,,\,\,\,\,\,\,  \dot{x}_1 =  x_2 - |x_1|^2  x_o   \nonumber  \\
\dot{x}_2 & = &   x_3 - (x_2 \cdot x_1) x_o  \\
\dot{x}_3 & = &  -(x_3 \cdot x_1) x_o + (x_2 \cdot x_1) x_1 - |x_1|^2\, x_2      .      \nonumber
\end{eqnarray}
Notice that, by the previous general results, the above system automatically implies
the well known nonlinear equation \eqref{oldcubic} for cubic splines.

\begin{remark} The above system of equations reproduces the system derived by Crouch and Leite  \cite{CrouchLeite1991, CrouchLeite1995} for cubic splines in the case of the $(n=2)$-sphere. Notice that our results imply, in particular, that this system is \emph{Hamiltonian} for any $n$. We shall come back to these equations for $S^2$ in the next section.
\end{remark}

\section{Dynamical analysis of splines on $S^2$ } \label{observations} 
In this section, we want to analyze the solutions of the hamiltonian equations \eqref{intrinsic} coming from the PMP in the particular case of $Q=S^2$. In this case, the $SO(3)$ symmetry plays an important role: we can perform symplectic reduction to decrease the dimension of the system. In order to simplify the reduction-reconstruction procedure, we will use a description of non-zero tangent vectors 
$$
TS^2 - 0 \simeq \R_+ \times SO(3) 
$$
which is based on the Gauss map for a convex hypersurface\footnote{The procedure could be generalized for convex n-dimensional hypersurfaces in $\Re^{n+1}$.}  and detailed in Appendix \ref{convexsurfaces}.   
In section \ref{CrouchLeite1} the dictionary between the above variables and the general split ones of Proposition \ref{mainproposition2} is described.   We also discuss the artificial singularity at $v=0$ introduced  by the above identification.

\subsection{Hamiltonian equations on $T^*(TS^2-0)$}  Let us consider a $2$-sphere of radius $r$ and take $M:=TS^2-0$ to be the manifold given by all non-zero tangent vectors to the sphere. Following appendix \ref{convexsurfaces}, the map
\begin{equation}\label{eq:TS2SO3}
\begin{split} 
\R_+ \times SO(3) & \to  M= TS^2-0\\  
(v,R) & \mapsto (x=R {\bf e}_3, {\bf v}= vR\, {\bf e}_1),
\end{split}
\end{equation}
where $\{{\bf e}_1,{\bf e}_2,{\bf e}_3\}\in \R^3$ denote the standard basis vectors, is a diffeomorphism (c.f. \eqref{Gauss}). (One should not confuse the scalar $v\in \R_+$ with notation previously used for vectors.) Since the underlying surface $\Sigma = S^2$ is a sphere, this diffeomorphism preserves the natural $SO(3)$-actions (on $M$ is given by the tangent lift of the action by rotations on $S^2$).  Following equation \eqref{state} in appendix \ref{convexsurfaces}, the state equations
(\ref{stateequation})
  for curves in $TS^2$ with non-zero velocity (i.e. lying in $M$) read
$$ \dot{v} = u_1 \qquad \mbox{and} \qquad \dot{R} = R X, $$
with
\begin{equation} \label{statesphere}  
X=X(v,u_2) = \left( \begin{array}{ccc}   0 & -u_2/v  & v/r \\  u_2/v &  0 & 0 \\ -v/r & 0  & 0  \end{array}      \right)   
\end{equation}
and $(u_1,u_2)\in \R^2$ being the controls.   Note that $u_1$ represents the tangential acceleration and  $u_2 = v^2 \kappa_g$,   where $\kappa_g$ is the geodesic curvature}.
The skew-symmetric matrix $X \in so(3)$ can be conveniently represented as  
\begin{equation}  \label{Omega}
\Omega = \Omega(X) =  ( 0\,,\, v/r\,,\, u_2/v )\in \R^3 \,\, .
\end{equation}
We recall that, for cubic splines, the cost function is  
\begin{equation}  \label{costcubic}  
C =  \frac{1}{2} \beta \, (u_1^2+u_2^2)
\end{equation} 
while for the time-minimal problem: $C=1$ and   $u = (u_1,u_2)$ is constraint by
\begin{equation} 
u_1^2+u_2^2 \leq  A^2 \,\,.
\end{equation}

\subsubsection*{Applying Pontryagin's principle} The cotangent bundle of the state space $M$ is 
$$T^*M= T^*\R_+ \times T^*SO(3) \simeq T^*\R_+ \times (SO(3)\times \R^3),$$
where we used the standard left trivialization of the cotangent bundle of $SO(3)$.
We denote by $a\in \R$ the conjugate to $v\in \R_+$ and  $(R,M_1,M_2,M_3)\in SO(3)\times \R^3$ the left trivialized covectors on the rotation group.
The parametric hamiltonian \eqref{eq:Hu} yields in this case
\begin{equation*}  
H_u = -C + a \ u_1 + \Omega \cdot {\bf M} =  - C + a \, u_1 + M_2\, v/r  + M_3 u_2/v  \,\,. 
\end{equation*} 
     Now, finding the optimal  Hamiltonian is a trivial task. For cubic  splines 
     \begin{equation}   \label{optimizers}
     u_1^* = a/\beta,\, u_2^* = M_3/(\beta v)  
     \end{equation}  so that
\begin{equation} \label{cubic}
  H_* =  \frac{1}{2\beta}\, \left( a^2  +    (M_3/v)^2  \right) +  M_2\, v/r   \,\,  .
\end{equation}
For time minimal splines, 
\begin{equation} 
(u_1^*,u_2^*) = \frac{A}{\sqrt{a^2 + M_3^2/v^2}}\, (a,  \frac{M_3}{v} )\,\,\,,\,\,\,  H_* =   A \sqrt{a^2 + M_3^2/v^2} + M_2\, v/r.
\end{equation}

\subsubsection*{Hamiltonian equations for reduction-reconstruction}
 The  symplectic structure on $T^*M\simeq T^*\R_+ \times (SO(3)\times \R^3)$ is the product of the canonical one on the first factor and the very well known one on the second factor (e.g. from the rigid body problem). It is then straightforward to derive the hamiltonian equations coming from $H_*$. Moreover, we observe that $H_*$ does not depend on $R$ so that it descends to a reduced hamiltonian function on 
$$
M_{red} = T^*M/SO(3) = T^*\R_+ \times \R^3.
$$
The induced Poisson brackets on $M_{red}$ are also well known 
$$ \{ v, a\} = 1, \ \{ M_i, M_j\} = \epsilon_{ijk} M_k.$$ 
The equations on $T^*M$ can be thus split into the reconstruction equations for $R(t)\in SO(3)$,
\begin{equation}\label{eq:recon}
 \dot{R} = R \ X_*, \ \ X_* = X(v(t),u_2^*(t))
\end{equation}
and the reduced equations for $(v,a,{\bf M})$:
\begin{equation} \label{optred}   
\dot{v} = \partial H_*/\partial a\,,\,  \dot{a} =  - \partial H_*/\partial v\,\, ,\,\,\, \dot{{\bf M}} = {\bf M}\times ({\rm grad}_{{\bf M}} H_*) \,\,.
\end{equation}
The function  
\begin{equation*}  M_1^2 + M_2^2 + M_3^2 = \mu^2 \,\, .
\end{equation*}
is a Casimir and restricts the dynamics of ${\bf M}$ to a momentum sphere.

For the rest of this section, we shall concentrate on the case of cubic splines.   In this case, $H_*$ is given  by \eqref{cubic} and, then, the reduced  equations \eqref{optred} read 
\begin{eqnarray}
\dot{v}   =   a/\beta, \ \dot{a}  =  - M_2/r  + M_3^2/(\beta \,v^3) , \
 \dot{M}   =  M \times (\, 0, \, v/r ,\, M_3/(\beta v^2) \, )    
\end{eqnarray}
  or, more explicitly, 
\begin{eqnarray}  \label{moreexpli}
\dot{v} & = & a/\beta\,\,\,,\,\,\, \dot{a} = - M_2/r  + M_3^2/(\beta \,v^3) \nonumber \\
  \dot{M}_1 & = & M_2 M_3/ (\beta v^2) -  M_3 v/r \nonumber  \\ 
   \dot{M}_2  & =  & - M_1 M_3/ (\beta v^2)  \label{reducedexplicit}\\ 
  \dot{M}_3  &= &   M_1  v/r. \nonumber
 \end{eqnarray}
  The study of time minimal case on $S^2$ will be submitted elsewhere \cite{KoillerStuchi}.

 \subsection{Dynamics of cubic splines from reduced system's fixed points}
 
 \subsubsection*{Equators: linearly accelerating geodesics } \label{equatorsaccelerating} The simplest fixed points for eqs. \eqref{reducedexplicit} in the virtual momentum sphere 
 correspond to the values  $M_1=M_3=0,  M_2=  \mu$ (we allow  $\mu$  positive or negative). Then, so to speak, the `poles' on the virtual momentum sphere are in the second coordinate $M_2$.  The variables $a$ and $v$ follow, respectively, a linear and a quadratic function of time:
 \begin{equation*}  a(t) = -  (\mu/r) \, (t -t_o) + a_o\,\,\,,\,\,\,  v(t) =  (\mu/2 r \beta) (t -t_o)^2 +  (a_o/\beta) (t-t_o) + v_o  \,\,.\end{equation*}
Since
$ \Omega = (0, v(t)/r , 0)
$
it is easy to reconstruct $R$ via \eqref{eq:recon}.  
In fact, setting  $R(t_o)= I$,   in the $x_1-x_3$ plane 
there is a  family of trajectories passsing at $ t=t_o$  trough the north pole of  the physical sphere (radius $r$):   
\begin{equation*}  \gamma(t)   =    r ( \sin(\theta(t)), 0 ,  \cos(\theta(t)))\in S^2 
\end{equation*}
with 
\begin{equation*}  \theta(t)   =   \mu/(6 r \beta) (t -t_o)^3 +   a_o/(2\beta) (t-t_o)^2 + v_o (t-t_o) +  \theta_o
\end{equation*} 
which honors the name ``cubic'' splines.
 
\subsubsection*{More fixed points of the reduced system}   
Fixed points can be parametrized by $v\in\R_+$ since stationary points of \eqref{reducedexplicit} must satisfy  $a=0, \,  M_2 =  r\, M_3^2/(\beta \,v^3)   $ and  $M $ parallel to $  (\, 0, \, v/r ,\, M_3/(\beta v^2) \, ) $.     A simple algebraic manipulation yields 
\begin{proposition}
For each  $v > 0$   there are two equilibria with  $\mu = \sqrt{2}\,\,\beta  v^3/r$ and
\begin{equation}  \label{fixed1} 
  a = 0,\,\,\,  M_1 = 0,\,\,\, M_2 = \beta  \frac{v^3}{r},\,\,\,  M_3 =  \pm \beta  \frac{v^3}{r}
\end{equation}
\end{proposition}
These fixed points correspond to  {\it relative equilibria}  in the unreduced system.   
From \eqref{optimizers}  we have $ u_2^* =  M_3/(\beta v)$ while from the general state equation on $TS^2$ one deduces $u_2^*= \kappa_g \, v^2$ where $\kappa_g$ denotes the geodesic curvature of the underlying curve $\gamma(t)\in S^2$ (c.f. appendix \ref{convexsurfaces}). Since
 $M_3 =    \pm \beta  \frac{v^3}{r} $ ,  we get  
\begin{equation*}     
 |\kappa_g| =  \frac{1}{r}  \,\,.
\end{equation*}
Recall that  on a sphere of radius $r$, the parallel of latitude $\theta$ has geodesic curvature   $ \kappa_g = \tan \theta/r$ and thus $ \theta =   \pi/4. $  Moreover, we observe that 

\begin{proposition} \label{figeight}  (Figure eights.) The reconstructed curves in $S^2$ with $R(0)=I$,  corresponding to the two equilibria parametrized by $v>0$ as above,  are
 two  {\it orthogonal (touching)  circles}  making a  $\textup{45}^\circ$ angle  with the equatorial plane. They are given by
 \begin{equation} \label{twin}
 \gamma(t)   =  r \left( \frac{\sqrt{2}}{2}\, \sin\theta ,\, \pm  \frac{1}{2}(1-\cos\theta)    ,\,  \frac{1}{2} (1+ \cos\theta) \right) \quad \mbox{and} \quad \theta =  \sqrt{2} \,\frac{v}{r} \,t \,\,.
 \end{equation}
\end{proposition} 
 \medskip

\begin{proof} Since 
 $ u_2^* =  \frac{M_3}{\beta v}\,\,,\,\,\, M_3 = \pm \frac{\beta v^3}{r}$ it follows that   $u_2^*  = \pm  \frac{v^2}{r}  $ and 
 $$ 
 \dot{R} =  R X_* \quad \mbox{with} \quad X_* = \left( \begin{array}{ccc}   0 &   \mp v/r   & v/r \\  \pm v/r &  0 & 0 \\ -v/r & 0  & 0  \end{array}      \right)  
 $$
 So we have steady rotations  with angular velocity $\omega =  \sqrt{2} \,v/r $ about 
\begin{equation}  
(u_x,u_y,u_z) = (0\, ,\, \frac{\sqrt{2}}{2} \,,\, \pm  \frac{\sqrt{2}}{2} ).   
\end{equation}
 Recall that for an unit vector $(u_x,u_y,u_z)$    the rotation matrix  $R(\theta)$  with $R(0) = I$ is given by
 {\small
$$\left[  \begin{array}{lll}
 \cos\theta + u_x^2 (1-\cos\theta) & u_x u_y  (1-\cos\theta) - u_z \sin\theta  & u_x u_z  (1-\cos\theta) + u_y \sin\theta   \\
u_x u_y  (1-\cos\theta) - u_z \sin\theta & \cos\theta + u_y^2 (1-\cos\theta) &
u_z u_y  (1-\cos\theta) - u_x \sin\theta \\
u_z u_x  (1-\cos\theta) - u_y \sin\theta & u_z u_y  (1-\cos\theta) + u_x \sin\theta &  \cos\theta + u_z^2 (1-\cos\theta)
\end{array}
\right]
$$ 
}
Equations (\ref{twin}) come from the third column of $R(\theta)$.  \end{proof}
  In  hindsight, we could  allow  $v<0$ in  (\ref{twin}), so we can describe both  twin circles in  both directions.  We have therefore {\it four} solutions,  each twin pair starting at the north pole   $(0,0,r)$    with velocity
  vector $(v,0,0)$.

  \subsubsection*{Discrete symmetries}  They  are in correspondence with expected geometric symmetries in the full system (in $S^2$).  i) Reflecting  a solution curve  $\gamma(t)$ over the equator that it is tangent 
  at a given point. Hence, given a solution, one gets infinitely many others (but two successive reflections correspond to the action of an $SO(3)$ element on the original curve).  ii)  Velocity reversal\footnote{It suggests that a double covering may be lurking around (perhaps
  $S^3 \to SO(3)$?).}   iii) Time reversal: it  implies that there is a symmetry between stable and unstable manifolds that  perhaps could be numerically explored.  

  \begin{proposition} \label{discrete}  Discrete symmetries.  
  \begin{enumerate} 
   \item[$(i)$] Reflection (Left-right): 
 {\footnotesize \begin{equation*} 
 \tilde{v} =    v(t)\,\,,\,\,\,  \tilde{a}(t) = a(t)  \,\,,\,\,\, \tilde{M}_1(t) =  -M_1(t) \,,\,\,   \tilde{M}_2(t) =  M_2(t)  \,\,\,,\,\,\, \tilde{M}_3 =  -M_3(t).
\end{equation*}}
     \item[$(ii)$] velocity reversal: {\footnotesize \begin{equation} 
 \tilde{v}(t) = -v(t)\,,\,  \tilde{a}(t) =  -a(t)  \,\,,\,
\tilde{M}_1(t) =  M_1(t)\,,\, \tilde{M}_2(t) =  -M_2(t)\,,\, \tilde{M}_3(t) =  - M_3(t). 
 \end{equation} }
    \item[$(iii)$]  Time-reversal:  
  {\footnotesize \begin{equation*} 
  \tilde{a}(-t)=-a(t), \tilde{v}(t)=v(-t) \,,\, \tilde{M}_1(t) =-M_1(-t), \tilde{M_2}(t)=M_2(-t), \tilde{M}_3(t)  =  M_3(-t). \end{equation*}
  }
  \end{enumerate}
  \end{proposition}
 
 These symmetries may be useful in finding  periodic orbits via global  calculus of variations,  and   perhaps find their stability using symplectic techniques \cite{Chenciner2000}, \cite{MLewis2013}.

 \subsubsection*{The fixed points are focus-focus singularities} 
In order to linearize about the equilibria it is convenient to take  spherical coordinates
 on the momentum sphere, 
\begin{equation}  \label{momentumsphere}  
{\bf M} =  \mu \, (\, \cfi \, \cte\,,\, \sfi\,,\,  \cfi \, \ste\,) \,\,.
\end{equation}
The reduced system is confined to the symplectic manifold  $M_\mu := T^*\R_+ \times S^2_{\mu}$,  where 
$S^2_{\mu}$ is the momentum sphere of radius $|\mu|$ (and recall that $T^*\R_+ = \{(v,a) \,: \, v > 0 \}$).

We can also define $ z= \sin \phi$ so that the symplectic form  on   $M_\mu$  becomes
\begin{equation} \label{omegared}   \Omega_{M_{\mu}}   =  da \wedge dv   + \,  \mu \, \cfi \,d\phi\,\wedge \, d\theta =  da \wedge dv   + \,  \mu \, dz \wedge \, d\theta
\end{equation}
and the reduced optimal Hamiltonian writes as (recall $z = \sin \phi$)
\begin{eqnarray} \label{hred}   H^{{\rm red}}_* &= &\, \frac{1}{2\beta} \,a^2 +  \,\frac{\mu^2}{2\beta} \, \frac{ (\cfi  \, \ste)^2}{v^2 } + \,\, \mu\, \sfi \, (v/r)\, \nonumber \\ & = & \frac{1}{2\beta} \,a^2 +  \frac{\mu^2}{2\beta} \,  (1-z^2)  \, (\ste)^2/v^2  + \, \mu\, z \, v/r
 \,.
\end{eqnarray}
  The equilibria are
\begin{equation} \label{fixed}
 a_o = 0\,\,\,,\,\,\, v_o^3 =  \, \pm\, \left( \frac{\mu \, r }{\beta} \right) \sqrt{2}/2   
 \,\,\,,\,\,\,   \theta_o =  \pi/2 \,\, {\rm or}\,\, 3\pi/2  \,\,\,,\,\,\, z_0 =  \pm \sqrt{2}/2 
 \end{equation}   with energy
 \begin{equation}  h^* = (3/2) \,\beta\,  (v^4/r^2) . \end{equation}

We add $v$ to the  parameters $\mu,r, \beta $.   It  turns out that the matrix  that linearizes the  Hamiltonian system given by  (\ref{omegared}) and (\ref{hred})   does not depend on $\mu$ and is the same for both equilibria:
\begin{equation}
A = \left(\begin{array}{cccc} 0 & -\frac{3\, \mathrm{\beta}\, v^2}{r^2} & -\frac{3\, \sqrt{2}\, \mathrm{\beta}\, v^3}{r^2} & 0\\ \frac{1}{\mathrm{\beta}} & 0 & 0 & 0\\ 0 & 0 & 0 & \frac{\sqrt{2}\, v}{2\, r}\\ 0 & \frac{3}{r} & -\frac{\sqrt{2}\, v}{r} & 0 \end{array}\right)
\end{equation}
 Furthermore, its characteristic polynomial does not depend on $\beta$.
\begin{equation}   p  =  {\mathrm{\lambda}}^4 + \frac{4\, v^2}{r^2}  \,  {\mathrm{\lambda}}^2\, +  \frac{12\, v^4}{r^4} .
\end{equation}

\begin{proposition}
The eigenvalues at the fixed points (\ref{fixed1})  (equivalently (\ref{fixed})) are loxodromic (focus-focus type) 
\begin{equation} (v/r) \, \sqrt{2}\,\, \sqrt[4]{3}\,\,  \left( \,\,  \pm  \sqrt{\frac{1}{2} - \frac{\sqrt{3}}{{6}} } \,\,\,\, \pm \,\,\,\,\,    \sqrt{\frac{1}{2} + \frac{\sqrt{3}}{{6}} }  \,\, i  \,\,\,  \right)
\end{equation} 

 In $T^* TS^2$,  the union for all $v \neq 0$   of these special circle solutions with $\kappa_g = 1/r$  forms a center manifold $C$ of dimension 4.
  \end{proposition}

 In the reduced space we have local unstable and stable (spiralling) manifolds
 of dimension two.   They lift to 6-dimensional stable and unstable manifolds $W_C^s, \, W_C^u $ inside $T^* TS^2$.  
This dimension count  is coherent with $  {\rm  dim } C =  6 + 6 - 8 = 4 $.
 
Several global dynamical question can now be posed:  on the unreduced system, take initial conditions 
near the focus-focus equilibrium. What happens with their solutions and  with the corresponding unreduced solutions?

More precisely, understanding the global behavior of $W^u$ and $W^s$ is in order. Do they intersect transversally?

\subsubsection*{Are equators in the `periphery' of  phase space?}  \label{periphery}

The equations of motion corresponding to the symplectic form \eqref{omegared}  and the Hamiltonian \eqref{hred}  are given by
\begin{eqnarray}   \label{reducedztheta}
\dot{v} & = & a/\beta \,\,,\,\,\,\,\,\,\,
\dot{a}  =  \mu \left(  -  z/r  + \, \frac{\mu }{\beta}(1-z^2)\,\frac{(\ste)^2}{v^3} \, \right),   \\
 \dot{\theta} &  = & \,  \frac{v}{r} \, -   \,\, \frac{\mu}{\beta} \,\, \frac{z\,(\ste)^2}{ v^2 } \,\,,\,\,\,\,\,\,\,
   \dot{z}  =   \,   \, \frac{\mu}{\beta}  \,\ste\,\cte \,\frac{(z-1)(z+1)}{v^2}  \nonumber   \,\,\, .
\end{eqnarray}
\smallskip

\noindent Assuming that  $v=0$ is a regularizable singularity (taking into account the various symmetries of Proposition \ref{discrete},     translated to these coordinates),  we have 
 $$  a,v, \theta \in \R \,,  \,\,\,\, |z| \leq 1 \, . $$

\noindent   The horizontal lines
$z = \pm 1$ are invariant, equivalent to  $M_1=M_3=0, M_2 = \pm \mu$.  We know from the previous discussion that reconstruction yields  the equators in the unreduced system.  The coordinate $a$ runs uniformly in time from left to right at $z=-1$ and from 
right to left at $z=+1$, namely  $a(t) = -{\rm sign}( z) \mu t/r  + a_o$.   As we expect,  $v$ is quadratic on time, with leading term $-{\rm sign}( z) \mu t^ 2/(2 r \beta ) $.    

As for   $\theta$, 
for  $ |t|  $  sufficiently large the second term in the equation for $ \dot{\theta} $    can be dropped out.   Thus for such  large  $ |t|  $    we have  $\theta(t) \sim
-{\rm sign}( z) \mu t^ 3/(6 r \beta)$.   

 This means that except possibly  at intermediate times, the horizontal  invariant $\theta$ lines in the plane $ (\theta, z)$ 
run in opposite ways\footnote{This  information  could be of interest for symplectic topologists:  Poincar\'e-Birkhoff  theorem should be applicable. } for  $z=\pm 1$.      

\subsection{Simulations of  $S^2$ cubic  splines}  
\label{simulations}  Numerical work  on sphere splines include (we apologize for  omissions), \cite{Huper2006}, \cite{Noakesspherical2006}  (a survey for the computational geometry community),  \cite{Samir2012}  (a gradient descent method).  
Here  we present some experiments using the reduction to
two degrees of freedom.  Besides the `figure eights' of Prop.\ref{figeight}, the  other family of solutions that seems to have fundamental dynamical importance are the equators, described in \ref{equatorsaccelerating}.

  \begin{figure}[b]   
  \label{integrablesection}
   \includegraphics[width=0.8\linewidth]{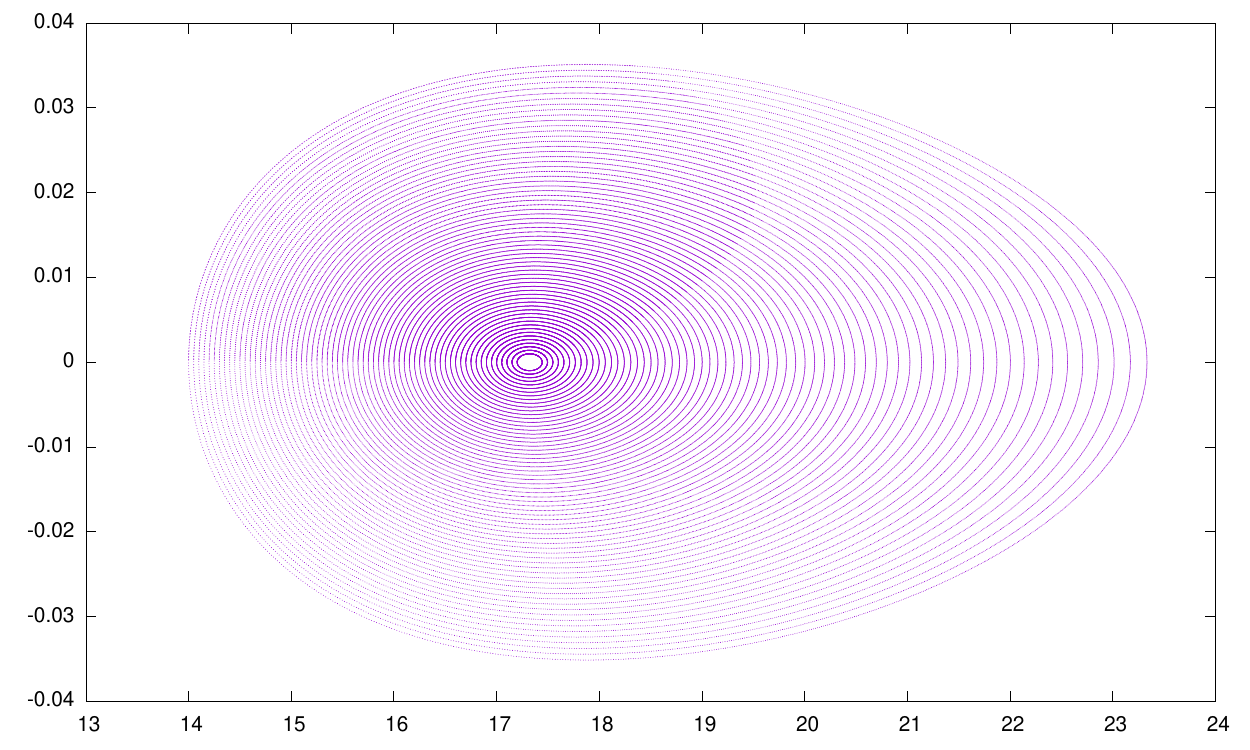}
   \caption{Energy  $h = 0.01$.   Regular trajectories.}
\end{figure}

 \begin{figure}
 \label{chaoticsection1}
   \includegraphics[width=0.8\linewidth]{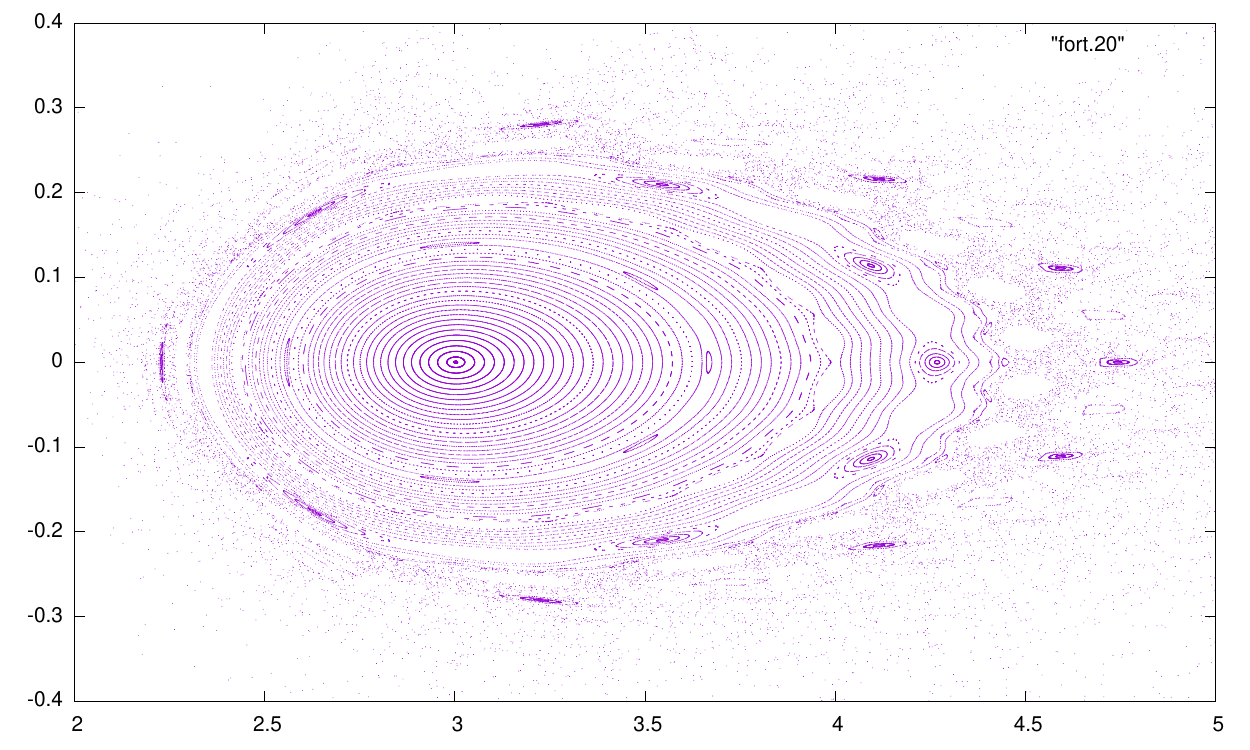}
   \caption{Energy  $h = 0.332412099$.  There is a large chaotic zone, with escaping trajectories.}
\end{figure}

 \begin{figure}
 \label{chaoticsection2}  
   \includegraphics[width=0.8\linewidth]{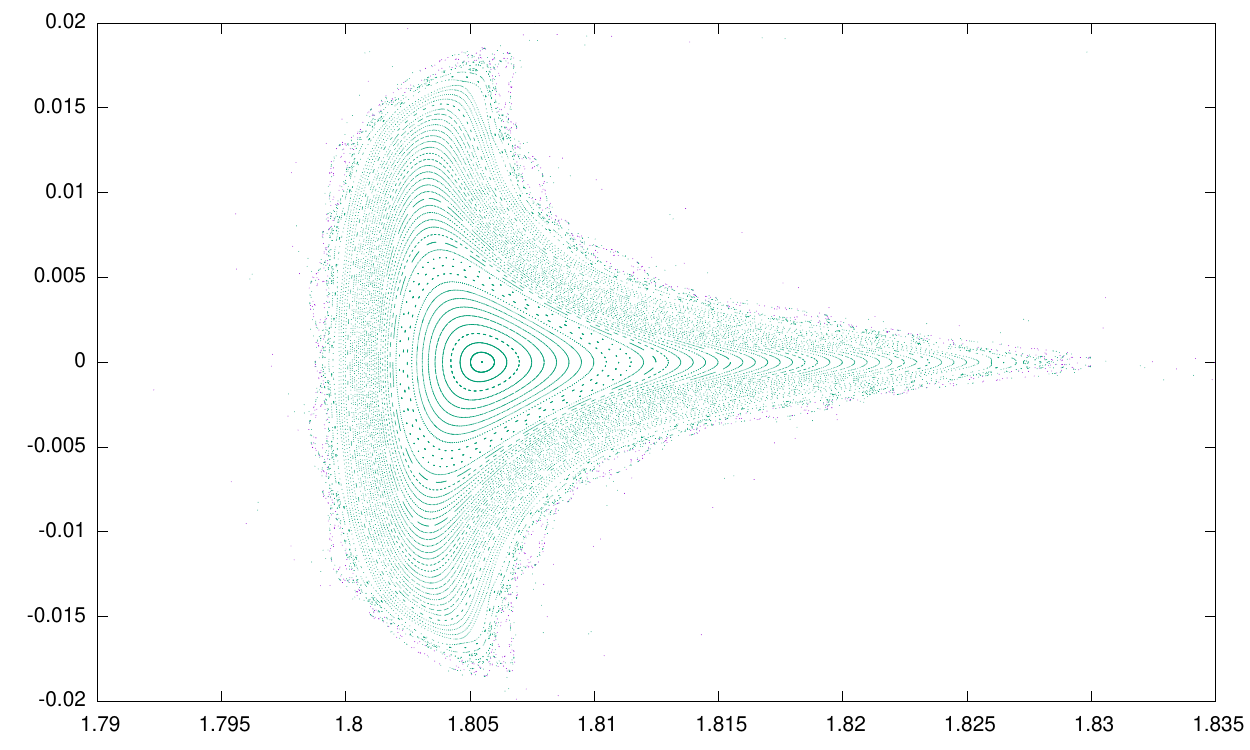}
   \caption{Energy  $h = 0.806$. Even  larger chaotic/escaping zone. The triangular feature is probably
   related to a 3:1 torus resonance. Note that we zoomed in with respect to Fig.2.} 
   \end{figure}
 
  \begin{figure}
 \label{chaoticsection3}
   \includegraphics[width=0.8\linewidth]{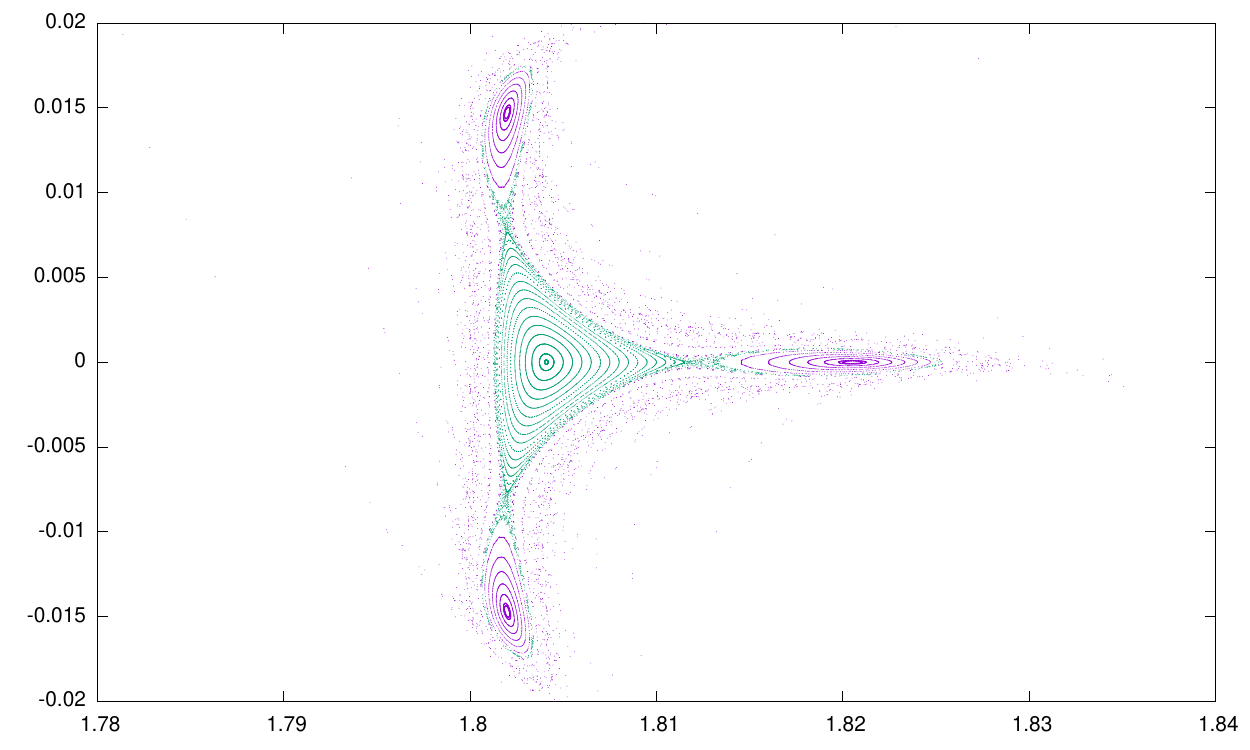}
      \includegraphics[width=0.8\linewidth]{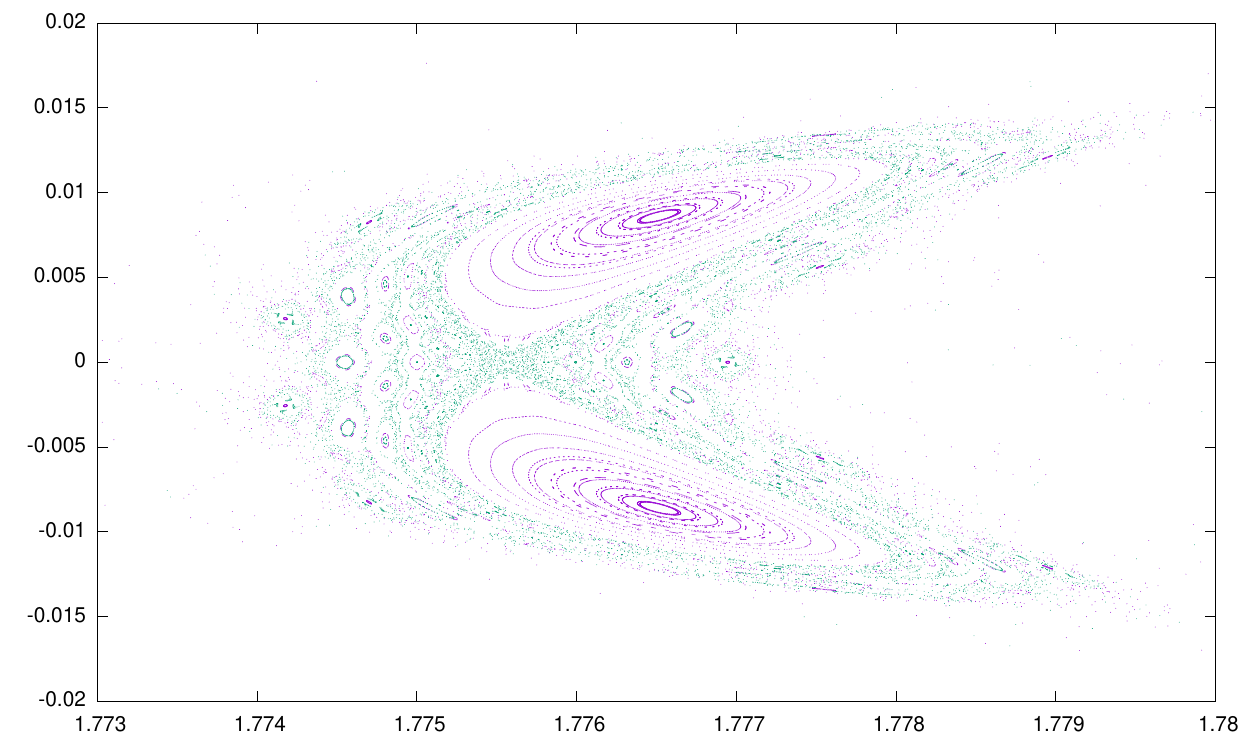}
         \includegraphics[width=0.8\linewidth]{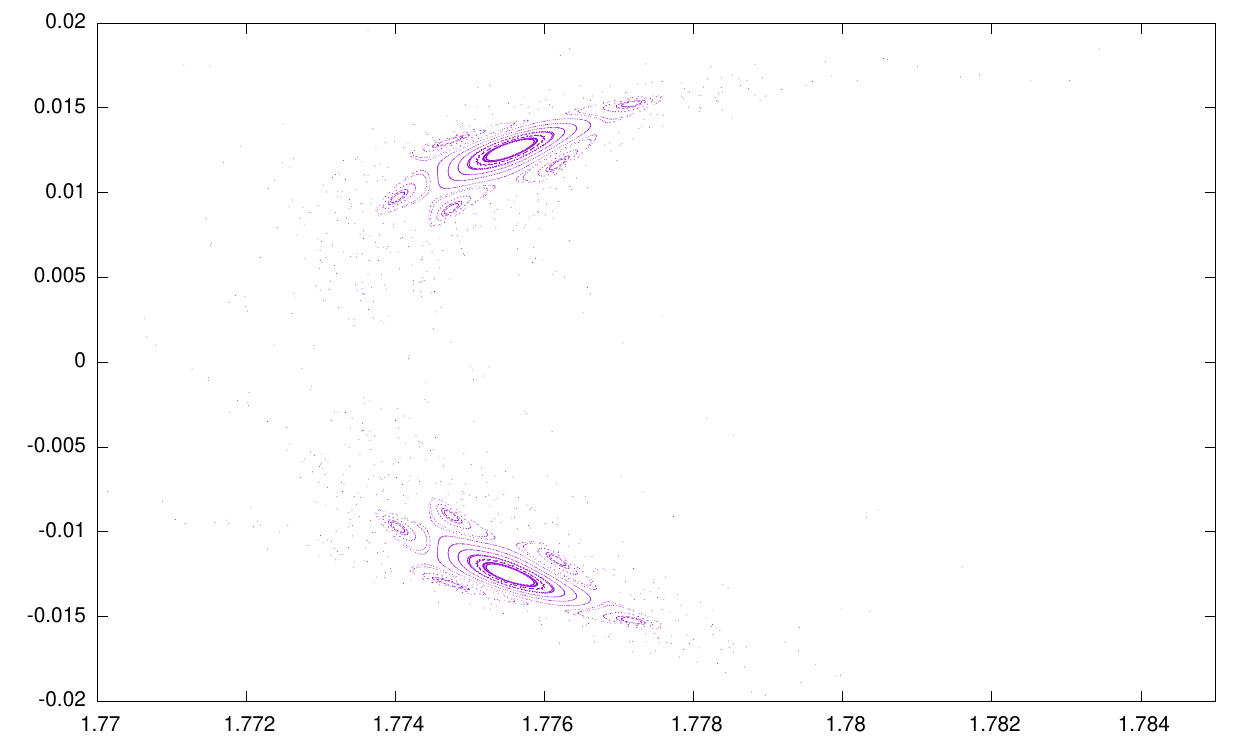}
   \caption{Nearby energies   $h = 0.8065, \,0.818, \, 8189$.  Only a small $a$ interval was depicted for better visualization.  Which bifurcations took place:  pitchfork, period doubling, Hamiltonian Hopf? 
   }  
\end{figure}

  \begin{figure}
 \label{trajectorytorus}
   \includegraphics[width=0.8\linewidth]{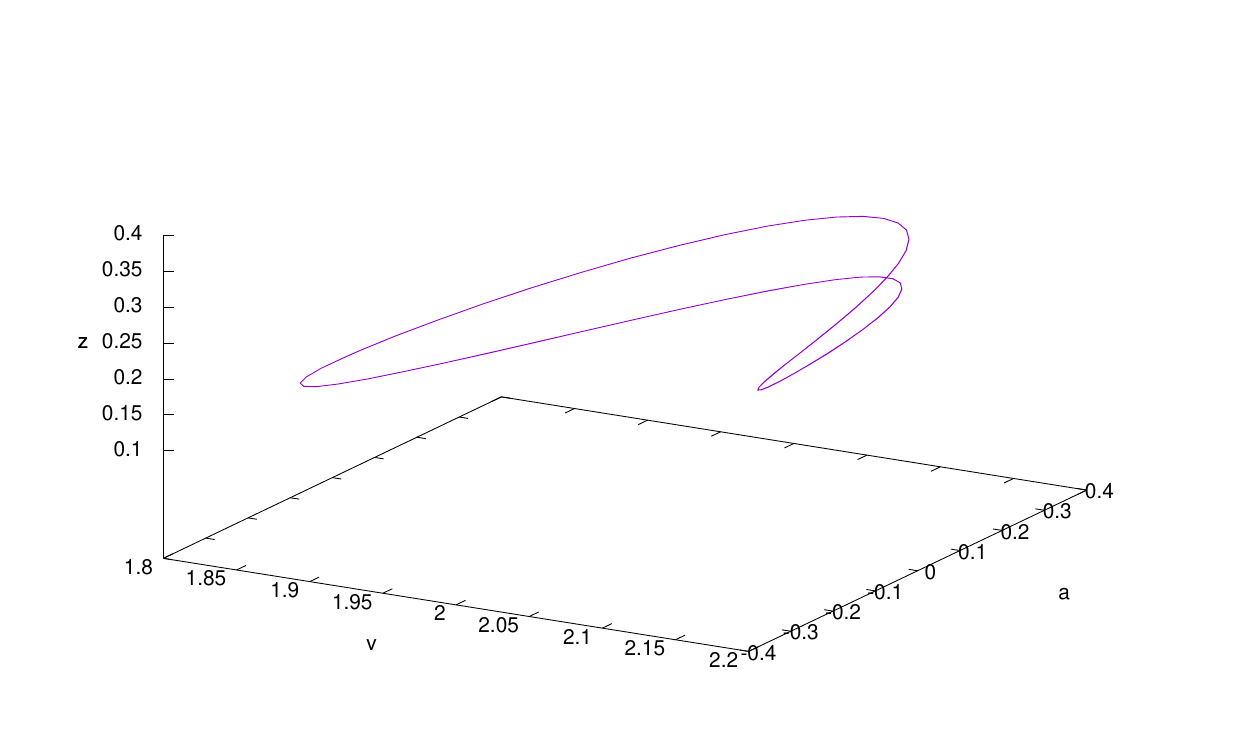}
      \includegraphics[width=0.8\linewidth]{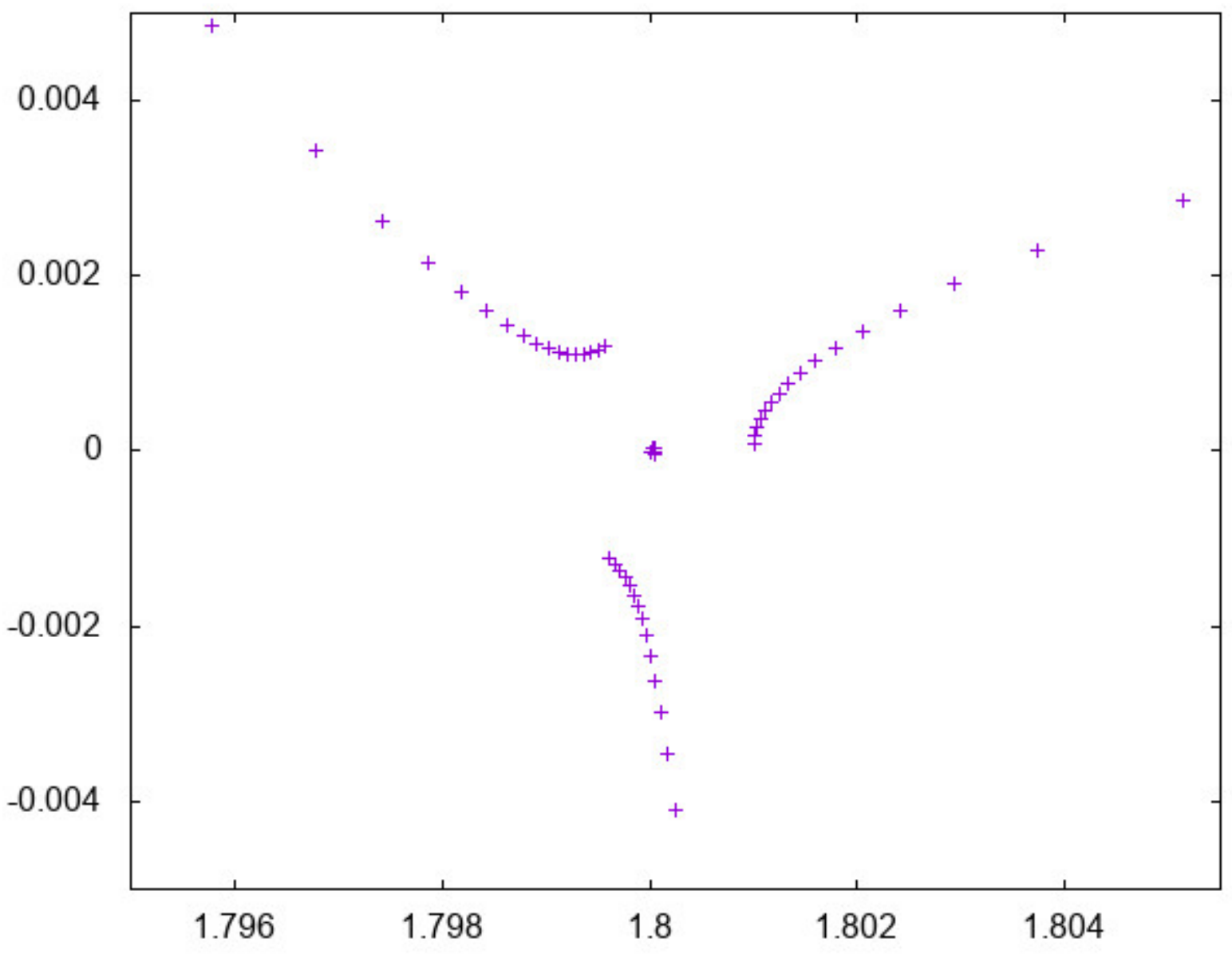}
   \caption{The periodic trajectory in the energy level   $h = 0.808$.  Note the central zone shrinking in the associated surface of section} 
   \end{figure}

Understanding  the dynamics near the equators is important both  conceptually and
 numerically  (see figs 6.1 and 6.2 in \cite{Noakes2014a}). Linearization does not help. It  is easy to see from the general cubic spline equation \eqref{oldcubic} that   for any Riemannian metric   geodesics whose accelerations vary linearly   are    $L^2$ splines.
 Our original (uninformed)  guess  was that, for any Riemannian manifold, splines would tend to accelerating geodesics as 
 $t \rightarrow \pm \infty$.\\
 
 Our numerical experiments indicate that this is not the case for cubic splines on $S^2$.     They strongly suggest that the system is non-integrable.   However,  we found zones in the reduced phase space having
 invariant tori.  
 
 \subsubsection*{Surface of sections}  
In Figures 1 to 4  we depict some Poincar\'e sections  of phase space $(v, a, \theta, z)$, taking 
 $\theta = \pi/2$.  The vertical axis is $a$, horizontal $v$.   The parameters are  $\beta= 1,  r = 2, \mu =2$.  Figure 5 shows a central periodic trajectory.

 \subsubsection*{Invariant tori}  Figs.6 and 7  depict  invariant tori.  
 They were found,  among several others with complicated Lagrangian projections,  by (obsessive) trial and error experimentation.  They live on remote regions of phase space, neither close to the loxodromic equilibria nor to the equators $z=\pm 1$.   There is a variety of confined trajectories whose `morphology'  merit further study.

  \begin{figure}
 \label{invarianttorus}
   \includegraphics[width=0.8\linewidth]{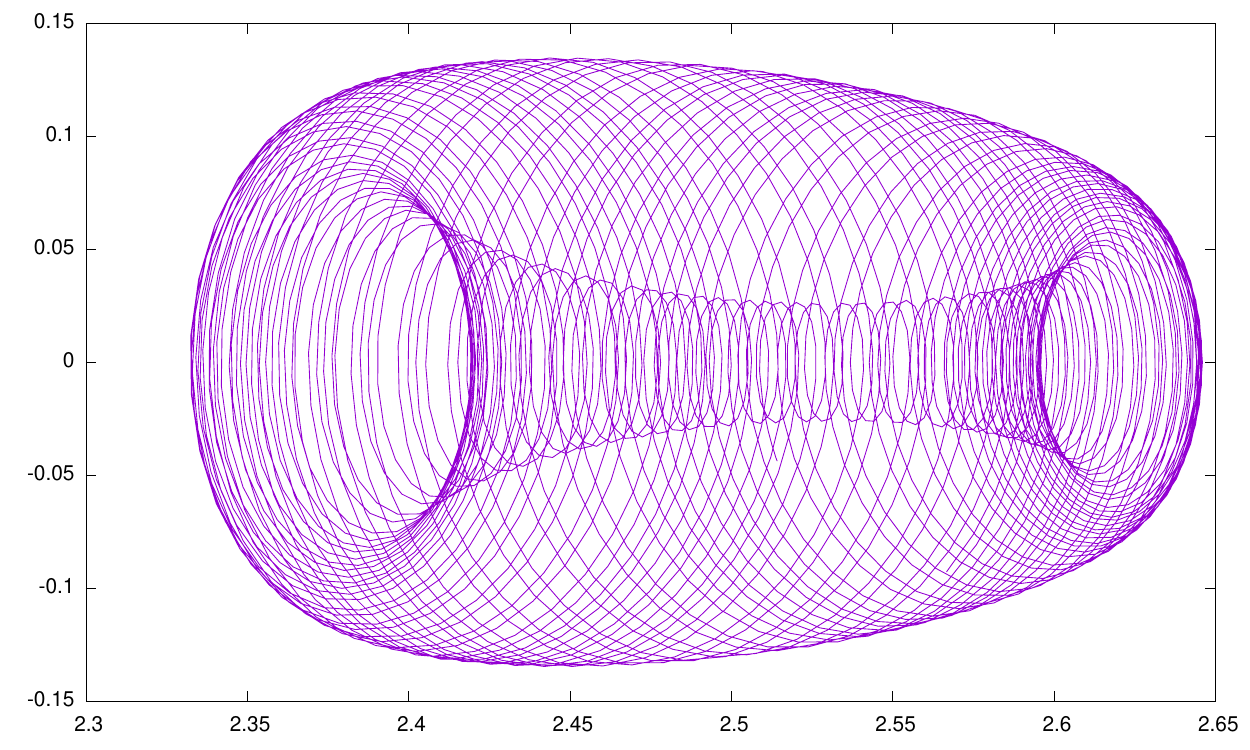}  
    \includegraphics[width=0.8\linewidth]{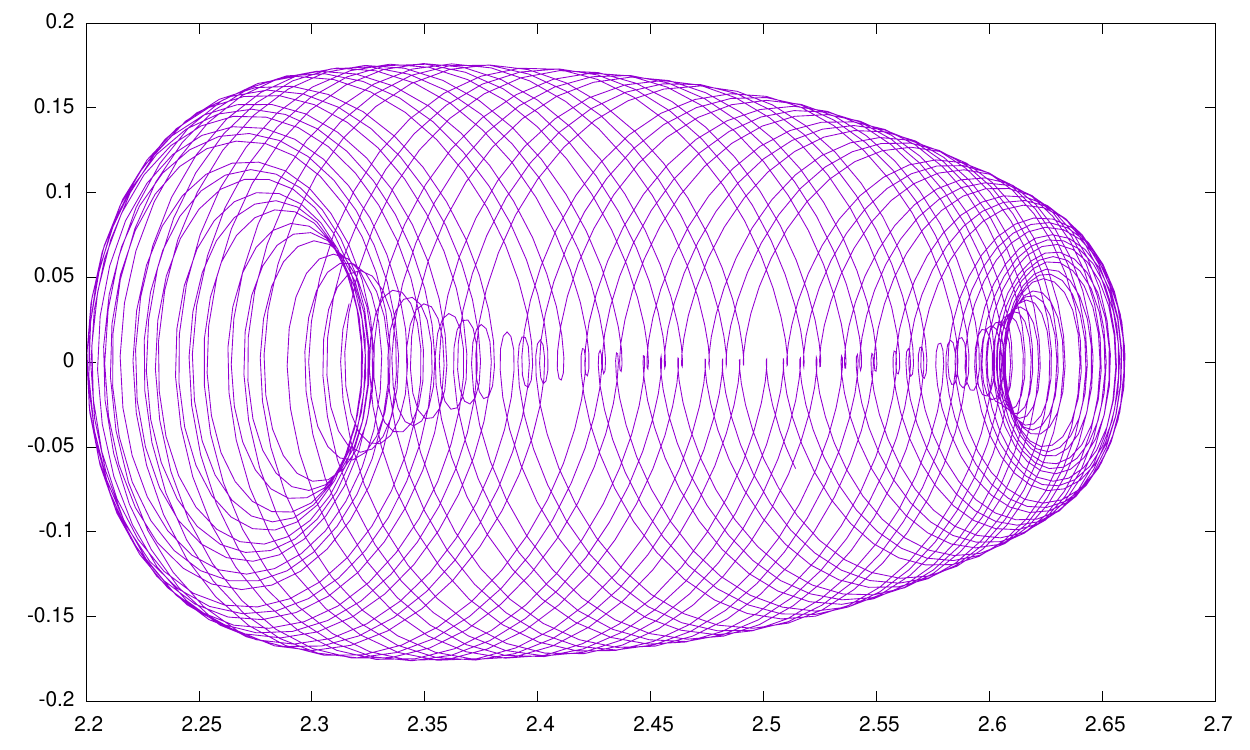}  
   \caption{Invariant tori, seen on a Lagrangian projection in the plane $(a,z)$. 
   Energies\,  $ h=0.49494873, \, $ and  $h = 0.522397316$.
}
\end{figure}

  \begin{figure}
 \label{invarianttorus2}
   \includegraphics[width=0.8\linewidth]{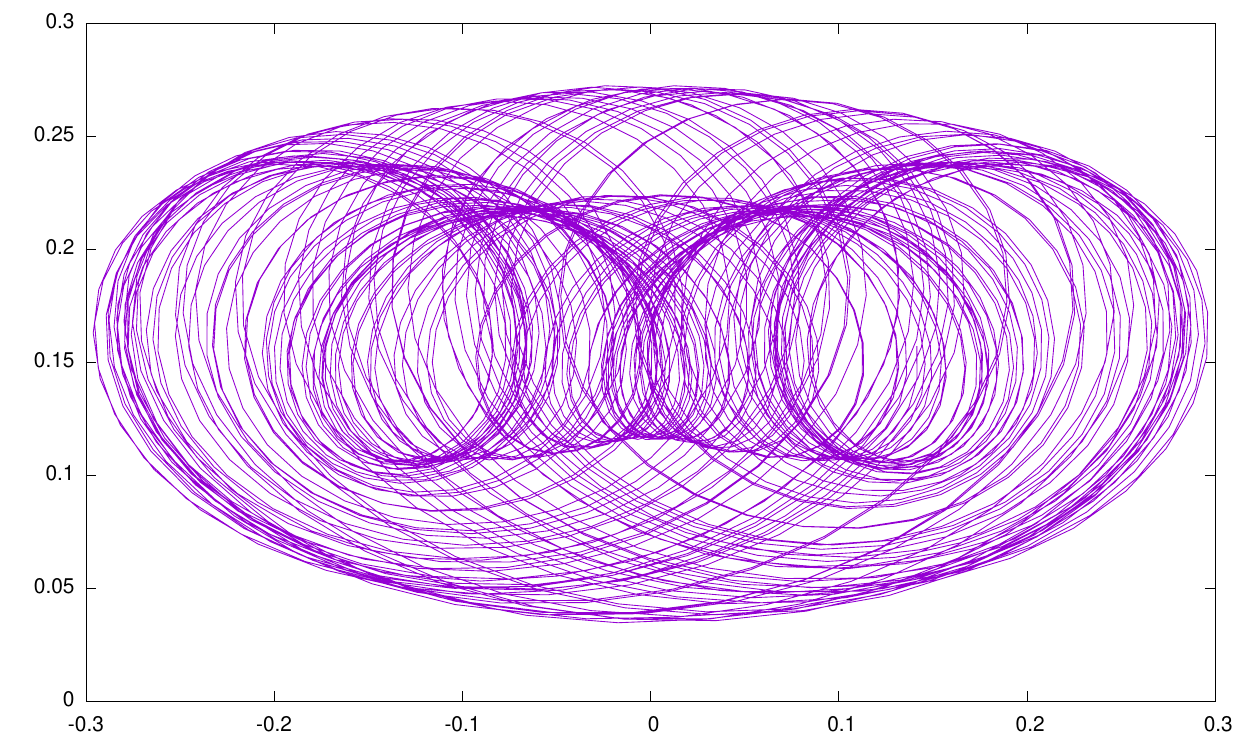}  
   \caption{An invariant torus, seen on a Lagrangian projection in the plane $(a,z)$. 
   $h = 0.586204019$, $\beta=1, \mu= r= 2$. }
\end{figure}

 \subsubsection*{Trajectories emanating from the focus-focus}  The simulations suggest that  (at least some of them)  are  reaching a neighborhood of an equator.    Will they eventually recur back to the focus-focus loxodromic equilibrium (figure eights of the unreduced)?   Our simulations suggest that  as $ t \rightarrow \infty$   those trajectories seem to `orbit'  around the equator
 $z=-1$, but they stay at a `safe' distance to it  (see Figure 8).        
 
 We assert  that  this is {\it not} an artifact of the integrator.  Here's an heuristic argument.    From (\ref{reducedztheta}) it follows that  when $v \rightarrow \infty$  then  $\dot{\theta} \sim v/r $, 
 and $\theta$  also diverges (only more so).   Therefore   $\dot{z} = O(v^{-2})  \rightarrow 0$
  is an ever  oscillatory way.  This suggests that  $z$  perhaps could stabilize close to  $z= -1$, but at a distance of it. 
  
  A study of  (\ref{reducedztheta}) is in order.   The behavior of the systems for moderate values of $v$ seems quite unpredictable.

  \begin{figure}
 \label{trajectory0}
      \includegraphics[width=1.\linewidth]{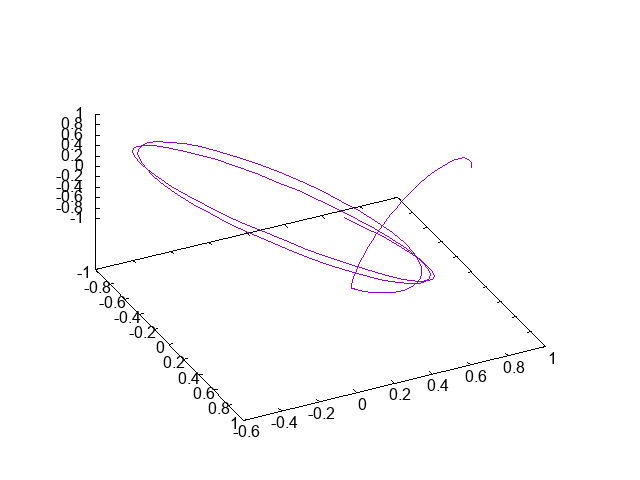}
      \includegraphics[width=0.8\linewidth]{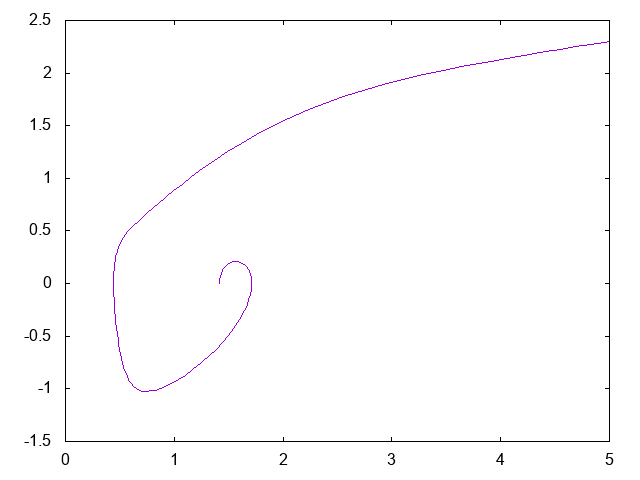}
   \caption{Top: reconstructed  trajectory in the physical sphere, that approaches a neighborhood of an equator.  Below: the  reduced trajectory  emanating from the  unstable equilibrium, projected in the $ (v,a)$ plane.  Note that  $v$ is growing quadratically with respect to $a$. 
   The reconstructed trajectory is approaching a neighborhood of an equator.  It remains to be  seen if it stays there or returns to a vicinity of the reduced equilibrium. } 
   \end{figure}

\subsection{Zero velocities and the relation to the general approach via split variables} \label{CrouchLeite1}

\subsubsection*{Motivation: unfolding $v=0$}
The price we had to pay in the description given so far of cubic splines in $S^2$ is that the scalar velocity $v$ appears  in denominators of
(\ref{moreexpli}), and it  may vanish   along a solution in finite time.  Nonetheless,  we expect that this is regularizable. In other words,  we posit that troubles at  $v=0$  are  (unfortunate) artifacts of our parameterization of $TS^2$ which excludes zero velocities and, ultimately, of the reduction procedure we implemented. (Notice that the lifted $SO(3)$ action on $TS^2$ is not principal; it is so when restricted to $M=TS^2-0$.)

We must then  go  back to  the  beautiful  
 equations \eqref{CLn} for the unreduced system (taking $n=2$ there). As mentioned earlier, these were originally derived by Crouch and Leite  \cite{CrouchLeite1991,CrouchLeite1995} (we also obtained them from our general split variables recipe in section \ref{subsec:gensplinesspliteqs}). 
 Clearly,  unreduced solutions  pass through ${\bf v} = x_1=0$ as smoothly as anywhere else.  

Regularization consists in lifting a reduced solution such that $v(t_o) = 0$  to an unreduced solution.  In Crouch-Leite system there is no trouble passing though $x_1(t_o)=0$,  and projecting back the
continued unreduced solution\footnote{Moreover,  the numerical integration of (\ref{moreexpli})  and the reconstruction of  equation $\dot{R} = R X$ will most likely  perform  disastrously when  $v$ approaches zero, so in practice it may be better, anyway,  to integrate numerically the unreduced system.}.  In Proposition \ref{Poisson} below we will provide the explicit  Poisson map relating the unreduced split variables $(x_o, x_1, x_2, x_3)$  to the reduced ones $(a,v, M_1,M_2,M_3)$.  

\subsubsection*{Regularizing the reduced systems directly?}   Before moving on to relating the unreduced and reduced system, we mention the following heuristic argument.  Take the generic situation that    $x_2(t_o) \neq 0$  when  $x_1(t_o) = 0$.  Then
\begin{equation*}   x_1   \sim  x_2(t_o)  (t -t_o)  \,\, {\rm for}\,\,  t  \, \,{\rm near}\,\, t_o.  \end{equation*}
This has a dramatic consequence  in our reduction: $e_1 = x_1/|x_1| $  sudenly flips from  $ - x_2(t_o) /| x_2(t_o) |$ to  $+  x_2(t_o) /| x_2(t_o) |$.  In this case we argue that  the continuation of  (\ref{moreexpli})   beyond  $t_o$  could be done using the velocity reversal symmetry of the problem. 

One indication is the behavior of the infinitesimal rotation
(\ref{Omega}) as  $t \rightarrow t_o^-$.    We have $\Omega = (0, v/r,  u_2/v) \rightarrow  (0, 0, \infty)\,\,\,. $    

We expect that this ``infinite infinitesimal rotation'' around $e_3$  will amount to an instantaneous rotation
by $\pi$ of the tangent plane.  In other words, the vectors $e_1,e_2$  will  instantaneously change sign at $t_o$, that is,  $$e_1^+ = - e_1^-,\, e_2^+ = - e_2^-\, .$$ 

Indeed, in view of (\ref{controls})
\begin{equation*}  \int_{t_o^-}^ {t_o^+} (u_2/v) dt =  \int_{t_o^-}^ {t_o^+} \,( \kappa_g v )\, dt  =  \int_{s_o^-}^ {s_o^+} \kappa_g \, ds
\end{equation*}
and  we posit that $\kappa_g$ should  be  $\pi$ times the delta function at $s_o = s(t_o)$.\\

\subsubsection*{The Poisson map   $(x, {\bf v}, p, \alpha) \to   (a,v, M_1,M_2,M_3)$ }
We need to compute the composition 
$$ q^*(T^*S^2)\oplus_{(TS^2-0)} q^*(T^*S^2) \overset{(i)}{\simeq} T^*(TS^2-0) \overset{(ii)}{ \simeq} T^* M \overset{(iii)}{ \to} M_{red}.$$

The symplectomorphism $(i)$ is the $\psi_\na$ of section \ref{subsec:splitvar}, with underlying change of variables given in proposition \ref{splittingsphere} (in the particular case of the $(n=2)$-sphere). The symplectomorphism $(ii)$ is induced by the diffeomorphism \eqref{eq:TS2SO3} between $TS^2-0$ and $M=\R_+ \times SO(3)$ (see also appendix \ref{convexsurfaces}). The Poisson map $(iii)$ is just the projection which forgets the $R\in SO(3)$ (once we use left trivialized covectors: $T^*SO(3) \simeq SO(3)\times \R^3$). \\
A straightforward tracking the above maps yields our final result (for simplicity we took $r=\beta=1$).
\begin{proposition}  \label{Poisson} 
The above Poisson map taking the split variables  $(x,\,{\bf v},\,p,\, \alpha)$  for  $ T^*(TS^2)$, as described in proposition \ref{splittingsphere}, to the reduced variables  $(a,v,M_1,M_2,M_3)  \in M_{red}=T^*\R_+ \times \R^3 $ is given by
 \begin{eqnarray} \label{reduction2}   
 a  & =  &   \alpha \cdot  {\bf v}/v  \,\,\,,\,\,\,   v = |{\bf v}| \nonumber \\
 M_1 & =  &    {\rm det}( p, {\bf v}/v, x )     \nonumber \\
M_2 & = &  \,\,\,    p \cdot {\bf v}/v    \\
M_3 & = &    {\rm det}(\alpha,\,x,\, {\bf v}) . \nonumber
\end{eqnarray}
\end{proposition}

\section{Comments  and further questions} \label{conclusions} 

\subsubsection*{Control  systems on anchored vector bundles}  \label{doublebundles}  More generally than in \eqref{stateequation}, which is a control problem with state space $A = TQ$, one could consider control problems with state space $A\ni (x,a)$ being a vector (or, more generally, affine) bundle  $q:A \to Q$ with a connection $\nabla$,  and state equations of the form
$$ \dot{x} = \rho(a), \ \nabla_{\dot{x}} a = u$$
with $\rho: A \to TM$  a given ('anchor') map. In this paper, we have been considering the particular case $A=TQ$ and $\rho=id$. Also note that for $u = 0$ (uncontrolled problem)  we recover the geodesic equations relative to $(\rho,\na)$.

Examples of such systems arise from nonholonomic control problems  \cite{Bloch}  and control on (almost) algebroids
\cite{MJozi}  (for background on algebroids, see  also \cite{paulette}, \cite{alan1},  \cite{manoloetal}).   We observe that control problems with state space $TQ$ with a Levi-Civita connection
can be recast, via the dual connection, to a control problem with state space $T^ * Q$.  This is a usefull observation for landmark splines,  since the problem is best described in terms of
a cometric.

In general, the PMP leads to the cotangent bundle $T^*A$ and to the problem of finding the optimal hamiltonian. The connection $\na$ allows us to obtain split variables generalizing those of section \ref{subsec:splitvar},
$$ T^*A \simeq q^*(T^*Q)\oplus_A q^*A^* .$$
Proposition \ref{mainproposition1} generalizes to this more general setting and provides a formula for the symplectic structure in split variables (also containing curvature terms). Moreover, the optimal hamiltonian is easy to find in these variables just as in \ref{subsec:splitvar}. This will be detailed in \cite{AlePoiJK}.

 \subsubsection*{Higher order splines and natural curvatures}  
For future work one can think of higher order splines as a  control  problem with state equation corresponding to
$$ \na^{(k+1)}_{\dot{x}} \dot{x} = u .$$
The state space can be taken to be $A=J^kQ$, the space of $k$-jets on $Q$. The  cost functional can depend on $u$ and (possibly) lower order covariant derivatives $D^{(i)}_{\dot{q}} \dot{q}\,,\, i = 1, \cdots k$. An optimal curve $\gamma(t)$  should connect two prescribed  $k$-jets    $j^{(k)}|_{x_o},\, j^{(k)}|_{x_1}\,\,$ (in computational anatomy it is often  required that  $\gamma(t)$
passes through  a  number of intermediary points at prescribed times). In this paper, we have treated the case $k=1$.

We remark that such a system is related to two very nice papers by  Gay-Balmaz,   Holm, Meier,   Ratiu,  and  Vialard \cite{Gay-Balmaz2012a,Gay-Balmaz2012} on higher order lagrangians. 

Our general approach via PMP applies here as well. One obtains $T^*(J^k(Q))$ and has to find the optimal hamiltonian. At this point, one notices that $J^k(Q)$ fits into a tower of affine bundles
$$ \cdots  J^{k}(Q) \rightarrow  J^{(k-1)}(Q) \rightarrow \cdots  \rightarrow J^1(Q) \rightarrow Q,$$ 
see \cite{Lewisaffine}. One can thus find split variables for  $T^* (J^{k}(Q)) $ recursively using a given connection on  $J^1 Q = TQ$. The higher order curvatures of the underlying curve $\gamma(t)\in Q$ will be related to components of the control, similarly to what we obtained in appendix \ref{convexsurfaces} for $Q$ a surface (and $k=1$) and, more generally, to what happens in \emph{elastica} (see e.g.  \cite{langer1984}, \cite{jurdjevic2014,jurdjevic2016}).  We plan to pursue this in a sequel paper.

\subsubsection*{Is accessibility  an issue?}  For the general context of accessibility in mechanical control problems, see \cite{barbero2011}.
In Alan Weinstein's  Ph.D. dissertation \cite{Weinstein1968},  about cut and conjugate loci  on Riemannian manifolds,  there is basic lemma  stating  that,  if the manifold is complete, connected, and of finite volume, then any two unit tangent vectors can be joined  by a smooth curve,   parametrized by arc length, 
with geodesic ends, having   geodesic curvature {\it smaller than any arbitrarily small bound.}  Using this  result it is easy to show accessibility for the time-minimal, bounded acceleration problem.  We wonder if the same is true in higher order, namely joining two given 2-jets by 
a curve with arbitrarily small ``jerk''.

\subsubsection*{Controllability on   vector bundles}  In the seminal paper by Lewis and Murray \cite{lewismurray}
on  configuration controllability of  mechanical systems, the concept of  symmetric product  of vector fields was introduced. Their results were  extended to mechanical systems with constraints and symmetries, see \cite{cortesetal}. For a geometric interpretation, see \cite{barbero2012}.  Can  the techniques  be used  in the  general context of control problems on  vector bundles with connection?  Note that the control  appears in  a fraction  ${\rm rank}(A)/({n+{\rm rank}(A)})$  of the equations (further, the  system can be sub-actuated).
 For  results on controlabillity of affine connection mechanical systems, see \cite{BarberoLinan2010}.

\subsubsection*{Diffusion PCA} More generally, on any framework where the phase space is a cotangent bundle $T^*P$ of a manifold with a bundle structure $P\to B$ (and a connection), a splitting of variables will be  useful. In the case of principal bundles  $G \to P \to P/G$,  the reduction of  $T^* P$ goes back to Kummer \cite{Kummer, Kummer1986} in the 1980's.
For instance,  a theory for diffusion principal component analysis (PCA) was developed by Sommer \cite{Sommer2015a},   based upon    stochastic development via  
Eells-Elworthy-Malliavin construction of Brownian motion \cite{Elworthy1988}, \cite{Sommer2015}.  A Hamiltonian system on the cotangent bundle of the frame bundle $T^* ({\rm Fr}(Q))$, governs the most probable paths\footnote{A code is available in
\url{https://github.com/stefansommer}.}.  
 
\subsubsection*{Interpreting the terms in  the (simple splines) Hamiltonian equations} 
Peter  Michor observed at the workshop that  the extra terms in the equation for $ \dot{p}_i$ in Proposition \ref{mainproposition1}   could be related to the concepts of   symmetrized force and  shape stress  in his work with Michelli and Mumford \cite{Michor2015}.

\subsubsection*{$L^{\infty}$ vs.  $L^2$}    For certain applications, Noakes has argued that $L^{\infty}$  could be better than $L^2$.
  Indeed,  from the mathematical side,  a drawback of  cubic splines  is that  for  manifolds of negative curvature  the velocity can become infinite in finite time 
\cite{PauleyNoakes2012}.    One can anticipate this behavior from equations  (\ref{mainequationsa}). They contain   curvatures -   signs  matter.  In contradistinction, under   bounded acceleration constraint,  the scalar velocity grows at most linearly, so in all cases  trouble is avoided by default.   Would that be physiologically reasonable?    Some simple experiments with $n=d=1$  shows that, for cubic splines, the acceleration can attain  high values of during the prescribed  time interval,  while the  time minimal bounded acceleration can do the job in not a much longer time, depending on the  concrete problem at hand.   For robotics applications,  or for an athlete, disastrous consequences could happen  if the norm of the control force  exceeds a  given bound at some instant, see  \cite{Ting2012},  \cite{Sohn2013}.

\subsubsection*{Singular reduction} 
\label{singular}   
$SO(3)$ does not act freely on $TS^2$:  trouble  happens  when the scalar velocity $v$ vanishes  (ie., the zero section of $TS^2$), and
this propagates to  non-freeness of  the action of $SO(3)$ on  the symplectic manifold $T^*(TS^2)$.   More generally,  one may consider  a vector bundle $A \to Q$  with a $G$-action,  that is not free on the base (hence on the zero section).    A procedure to do the singular Hamiltonian reduction of $T^*A$ is in order, a research direction that we hope to address in the future\footnote{Tudor Ratiu, Miguel Rodriguez-Olmos and Mathew Perlmutter are working out a general theory of singular reduction.  Their results for $T^*Q$, where the $G$ action on $Q$ is not free, should be
expanded to  $T^*(TQ)$.}. 

\subsubsection*{Applying Morales-Ramis theory}   Let us go back to cubic splines in $S^2$ as described in section \ref{observations}. Since the kissing circles unstable periodic orbits are explicitly known, one may hope to prove  nonintegrability  using the  Morales-Ramis approach \cite{Morales}.  However, linearizing (\ref{CLn}) and doing the required Galois theory for the time periodic linear equations would be, no doubt,  a tour-de-force.  On could also attempt to show nonintegrability linearizing around  the  equator solutions.

\subsubsection*{Controlled Lagrangians} In \cite{Jerryetal, Jerryetal2}  the concept of  {\it controlled Lagrangian} is introduced, for mechanical systems $L = T-V$ of natutal type.  The  control forces here keep the conservative nature of the controlled system. This is achieved by conveniently shaping the kinetic and/or potential energy.   The modified system  is still  a closed-loop system, and  the controlled system is Lagrangian by construction.  Energy methods are used to find control gains that yield closed-loop stability. It would be interesting to see if such methods could be used to match tangent vectors.

\subsubsection*{Splines in infinite dimensional Riemannian geometry}  This is a special edition  about  a meeting on infinite-dimensional Riemannian geometry,  so we now try to link the present work to the infinite dimensional setting. As it is customary, the idea is to use the present study of systems with underlying finite dimensional configuration spaces  $Q$ as simplified models for the cases in which $Q$ is an infinite dimensional Riemannian manifold. In this direction,    the Levi-Civita connection and the curvature of  Sobolev metrics  on $Q={\rm Diff} $ have been studied  by several authors,  see eg.   \cite{Bauer2014}, \cite{Mich2Mum}, \cite{harms},  \cite{Khesin2013,Khesinetal2013,KhesinWendt2009book},  \cite{MumfordMichor2013}. Below we enumerate some related questions.

\begin{enumerate} 
 \item[$(i)$] {\it PDEs for splines in shape space.} This means optimal control problems  with state space  $A = T {\rm Diff}$.  In order to  describe $T( T{\rm Diff})$ and  $T^*( T{\rm Diff})$,  one can take advantage of the fact that ${\rm Diff}$ is a group and, thus, $T{\rm Diff} = {\rm Diff} \times \mathcal{X}$.  What are  the corresponding  PDEs for  $L^2$ and  $L^{\infty}$ splines?  They should involve  not only  the momentum density ${\bf m}$  but  another density $\bf{p}$ corresponding to a Pontryagin multiplier (alternatively,   a PDE for ${\bf m}$ involving three time derivatives).

 \item[$(ii)$] {\it Lifting landmark splines.}  Consider splines on a finite dimensional landmark space,  i.e., with state space  $A = T (\R^d)^N $ and  cometric given by a Green function $G(x,y)$. In the case of a finite number $N$  of (point) landmarks, Mario Michelli  \cite{Micheli2012, Mich2Mum} has implemented the  geodesic equations for the landmark cometrics.
 In a   similar way as it can be done for EPDiff, can these be  lifted to solutions of a corresponding infinite dimensional spline problem on $T{\rm Diff}$?   {\it Faute de mieux},  one would use  the same ansatz (\ref{u}) to move other points in $\mathcal{D}$, but  in doing so we would be neglecting the new costate variables.   Once this question is elucidated, one could proceed to numerical discretization, see e.g. \cite{chertock2012}, \cite{Bauer2016},  \cite{pavlov}   for geodesics in ${\rm Diff}$.

\end{enumerate}

\bigskip

\noindent {\bf  Acknowledgements.}
Supported by Brazil's Science without Frontiers  grants  on Geometric Mechanics and Control,  PVE011-2012 and PVE089-2013.  We thank Darryl Holm, Tudor Ratiu and Richard Montgomery  and Alain Albouy for their generous participation in the project and Marco Castrillon for useful discussions.   JK  wishes to thank Martins Bruveris, Martin Bauer and Peter Michor for the invitation to  the  Program on Infinite-Dimensional Riemannian Geometry with Applications to Image Matching and Shape Analysis at the  Erwin Schrodinger Institute.

\begin{appendix}

 \section{Fortran program for reconstruction} \label{program} 

\footnotesize{

\begin{verbatim}
 
      implicit real*8(a-h,o-z)
      dimension z(13),b(13),f(13),r(13,13)
      common erk,amu,beta,rr
      external dertres
      erk=1.d-13
      n=13
 
c     parameters
      rr=2.d0
      amu=2.d0
      beta=1.d0
      pi=4.0*datan(1.d0)
 

c   variables are in order 1 to 13:
c   r13 r23 r33  r11 r21 r31 r12 r22 r32    v  a  tetha phi

c initial conditions
	 z(1)=0.d0
	 z(2)=0.d0
	 z(3)=1.d0
	 z(4)=1.d0
	 z(5)=0.d0
	 z(6)=0.d0
	 z(7)=0.d0
	 z(8)=1.0d0
	 z(9)=0.d0
	 z(10)=(amu*rr/(beta*dsqrt(2.d0)))**(1./3)
	 z(11)=0.0d0
	 z(12)=pi/2.d0
	 z(13)=pi/4.d0
c
      t=0.d0
      e=erk
      n=13
      h=.01d0
      hmi=1.d-8
      hma=.1d0
c	
      do i=1,200
	  call rk78n(t,z,n,h,hmi,hma,e,r,b,f,dertres)	  
	  write(20,*)z(1),z(2),z(3)
	  write(21,*)z(10),z(11),z(12),z(13)
      enddo
      stop
      end

C--------------------------------------------------------------------
C   FORTRAN SUBROUTINE FOR  EDOS  
C-------------------------------------------------------------------                
     

c  Runge Kutta code courtesy of C. SIMO group 

 subroutine dertres(a,b,n,f)
      implicit real*8(a-h,o-z)
      dimension b(13),f(13)
      common erk,amu,beta,rr
	 vr=b(10)/rr
	 vb=beta*b(10)**2
	 am3=amu*dcos(b(13))*dsin(b(12))
      f(1)=vr*b(4)
      f(2)=vr*b(5)
      f(3)=vr*b(6)
      f(4)=am3*b(7)/(vb)-b(1)*vr
      f(5)=am3*b(8)/(vb)-b(2)*vr
      f(6)=am3*b(9)-b(3)*vr
      f(7)=-am3*b(4)/(vb)
      f(8)=-am3*b(5)/(vb) 
      f(9)=-am3*b(6)/(vb)
      f(10)=b(11)/beta
      f(11)=-amu*dsin(b(13))/rr+amu**2*dcos(b(13))*dsin(b(12))**2/
     #	 (beta*b(10)**3)
      f(12)=(b(10)/(rr)-amu*(sin(b(12)))**2*dsin(b(13))/(beta*b(10)**2))
      f(13)=-amu*dcos(b(13))*dsin(b(12))*dcos(b(12))/(beta*b(10)**2) 
      return
      end
    \end{verbatim}
    }

\section{State equations on convex surfaces and the Gauss map} \label{convexsurfaces} 
In this appendix, we elaborate on a description of non-zero tangent vectors on a convex surface $\Sigma$ which uses the Gauss map. We use it to  provide an alternative form of the state equations \eqref{eq:generalstateeq} on $T\Sigma$. This description is used in section \ref{observations} in the particular case of $\Sigma=S^2$ to exploit the rotational symmetry.

 Let   $\Sigma \subset \R^3$ be a closed smooth convex surface.
The Gauss map  induces a diffeomorphism between $T\Sigma - 0$ and     $\R_+ \times SO(3)$: 
  \begin{equation}  \label{Gauss}   
    {\bf v}_q \,  \leftrightarrow    \,\, (v,R)\,\,\, ,\,\,\,\,\,  v =  ||{\bf v}_q || \neq 0
 \end{equation}
where $R \in SO(3)$ is constructed as follows: points  $q \in \Sigma$  correspond uniquely to  external unit normal vectors to the surface,   which we denote $e_3$.
 Now, a  nonzero tangent vector  ${\bf v}_q$   corresponds uniquely to a pair   $(v,  e_1)$  with  
  \begin{equation*}  
  {\bf v}_q =  v \,  e_1\,\,,\,\,  v  > 0\,\,\,  ,\,\, {\rm and}\,\,\,  e_1 \cdot  e_3 = 0 \,\,, \,\, |e_1|=|e_3| = 1 .
  \end{equation*}
 We use a redundant vector   $e_2 = e_3 \times e_1$  to construct the matrix  $R$ with columns $e_1,e_2,e_3$. 
 Therefore,  a control problem with state space  $T\Sigma$  corresponds to a control problem on  $SO(3) \times R_+$,  provided we exclude  the zero section\footnote{Therefore, it is important to characterize  which splines $\gamma(t)$   can have zero velocity at a certain time instant.  Are these splines non-generic?  At any rate,     laziness is not  expected on cubic splines:  $v$ should not vanish on an interval.
One expects (or at least hopes) that  $e_1$ can  be smoothly continued across $v=0$.   Some ideas are given 
section \ref{singular}.}.
 
Let us now move on to rewritting the state equations \eqref{eq:generalstateeq} in our present situation. Recall the Darboux formulas for a curve $\gamma(s)\in\Sigma$  ($' = d/ds$) 
$$  
e_1' = \kappa_g \,e_2 + \kappa_n\,e_3\,, \quad e_2' = - \kappa_g \,e_1 + \tau_g \,e_3\,, \quad e_3' = -\kappa_n \,e_1 - \tau_g\,e_2
$$
  where  $\kappa_g $ is the geodesic curvature,  $\kappa_n$ the normal curvature, and $\tau_g$ the geodesic torsion of
   $\gamma$. These formulas can be rewritten as  
  $$ 
  \dot{R} = R \, X \quad \mbox{with}\quad X = v\, \left( \begin{array}{ccc}   0 & - \kappa_g  &  - \kappa_n  \\   \kappa_g &  0 &  -\tau_g\\  \kappa_n    & \tau_g & 0  \end{array}      \right) .
$$
   The normal curvature is not freely controllable since it corresponds to the force that constrains the curve to stay in the surface.   Indeed, taking derivatives in the ambient space,
  $$  \ddot{\gamma} = \dot{v}\,e_1 + v^2 \,e_1' =  \dot{v} \,e_1 +  v^2 ( \kappa_g \,e_2 + \kappa_n\,e_3) = \nabla_{\dot{\gamma}}\, \dot{\gamma} + 
  v^2  \kappa_n\,e_3
  $$
with  $ \kappa_n = e_1' \cdot e_3 = - e_3'\cdot e_1 := B(e_1,e_1)$   where $B$ is the second fundamental form of the surface. On the other hand, the intrinsic description of the state equations, using the Levi-Civita connection, reads
  \begin{equation}  \label{nabla}  
  \nabla_{\dot{\gamma}}\, \dot{\gamma} = \, u_1 \, e_1  +  u_2\, e_2 
  \end{equation}
  where $u_1,u_2$ are the controls. The previous simple calculation thus showed that
\begin{equation} \label{controls}   
u_1 = \dot{v} \quad \mbox{and} \quad  u_2 =  v^2 \kappa_g.
\end{equation} 
 But the geodesic torsion $\kappa_g$ also admits the following interesting formula found by Darboux
  \begin{equation}  \label{torsion}  
  \tau_g =  \tau_g(e_1) = (\kappa_1-\kappa_2) \sin \phi \, \cos \phi
  \end{equation}
where  $\phi$ is the angle between the unit tangent vector $e_1$  to the curve  and a principal direction on the surface.  We then conclude that the state equations can be written as
 \begin{equation}  \label{state}   
 \dot{v}  =   u_1\,, \quad  \dot{R} = R \, X\,, \quad  
   X = \left( \begin{array}{ccc}   0 & -u_2/v  &   -v\, B(e_1,e_1)  \\  u_2/v &  0 &  -v\, \tau_g( e_1 ) \\      v\, B(e_1,e_1)    &  v\,\tau_g( e_1 ) & 0  \end{array}      \right) \,. 
\end{equation}
\vspace*{1mm}

 \end{appendix}

   \bibliographystyle{AIMS}

\bibliography{mybiblio}

\end{document}